\documentclass[a4paper,12pt]{article}
\usepackage[cp1251]{inputenc}
\usepackage[russian, ukrainian,english]{babel}
\usepackage{amsmath, amsthm, amsfonts, amssymb}

\theoremstyle{plain}
\newtheorem{thm}{Theorem}[section]

\newtheorem{lem}{Lemma}[section]
\newtheorem{cor}{Corollary}[section]
\newtheorem{rem}{Remark}[section]
\newtheorem{dfn}{Definition}[section]

\begin{document}
\large

\begin{center}
{\Large \bf It\^o-Wiener expansion for functionals of the Arratia's flow $n$-point motion}

\vskip5pt

Georgii Riabov

{\it Institute of Mathematics, Kyiv \\ National Academy of Sciences of Ukraine}

\end{center}

\begin{abstract}
The structure of square integrable functionals measurable with respect to the $n-$point motion of the Arratia flow is studied. Relying on the change of measure technique, a new construction of multiple stochastic integrals along trajectories of the flow is presented. The analogue of the It\^o-Wiener expansion for square integrable functionals from the Arratia's flow $n-$point motion is constructed.
\end{abstract}

\section{Introduction}

The present paper continues the study of orthogonal structure of square integrable functionals from coalescing stochastic flows undertaken in \cite{Dorogovtsev_Riabov}. The main object of our considerations is the Arratia flow on the real line. It is a family
$\{x(u,t):u\in \mathbb{R},t\geq 0\}$ of random variables such that

1) for every $u\in \mathbb{R}$ $x(u,\cdot)$ is a continuous square integrable martingale with respect to the filtration $\mathcal{F}^x_t=\sigma(\{x(v,s):v\in\mathbb{R},s\leq t\});$

2) $x(u,0)=u;$

3) $<x(u,\cdot),x(v,\cdot)>(t)=(t-\tau_{u,v})_+,$ where $\tau_{u,v}=\inf\{t\geq 0:x(u,t)=x(v,t)\}.$

The Arratia flow was constructed in \cite{Arratia}. Informally, it represents the motion of Brownian particles that start from every point of $\mathbb{R}$ and move independently until some of the particles meet each other. Thereafter these particles coalesce and continue their motion as  one particle. Given $u\in \mathbb{R}^n$ the $\mathbb{R}^n-$valued process $x_u(t)=(x(u_1,t).\ldots,x(u_n,t)),$ $ t\geq 0$ is called the $n-$point motion of the Arratia flow starting from $u$. It is a Feller process in $\mathbb{R}^n$ \cite[Prop. 4.2]{LeJan_Raimond}.

Development of stochastic analysis for the Arratia flow was initiated by A. A. Dorogovtsev in \cite{Dorogovtsev_integral} where the stochastic integral with respect to the Arratia flow was built. It was applied to prove analogues of the Clark representation theorem, the Girsanov theorem and to define the Fourier-Wiener transform for functionals of the Arratia flow \cite{Dorogovtsev_Clark, Malovichko, Dorogovtsev_Fourier}.  Let $\sigma_n$ be the moment when all particles $\{x(u,\cdot),u\in [0,1]\}$ have coalesced into exactly $n$ particles, i.e.
$$
\sigma_n=\inf\{t\geq 0: |x([0,1],t)|=n\}.
$$
Denote the trajectories of the $n$ particles on the interval $[\sigma_{n},\sigma_{n-1})$ by $(\eta^n_1,\ldots,\eta^n_n)$. By definition \cite{Dorogovtsev_integral}, the stochastic integral with respect to the Arratia flow is the following sum of one-dimensional stochastic integrals with respect to individual trajectories
\begin{equation}
\label{23_02}
\sum^\infty_{n=1} \sum^n_{l=1} \int^{\sigma_{n-1}}_{\sigma_{n}} a^n_l (t) d\eta^n_l(t).
\end{equation}
This construction and its applications in \cite{Dorogovtsev_Clark, Malovichko, Dorogovtsev_Fourier} rely heavily on the fact that individual trajectory $x(u,\cdot)$ is a Brownian motion.

For $n>1$ the $n-$point motion $(x(u_1,\cdot).\ldots,x(u_n,\cdot))$ is no longer Gaussian process. Moreover, the filtration generated by the Arratia flow is an example of a black noise in the terminology of B. S. Tsirelson \cite{Tsirelson} (i.e. the Gaussian part of the filtration is trivial). Understanding such filtration was the main motivation of the research undertaken in \cite{Dorogovtsev_Riabov}. In this paper we investigate the structure of square integrable random variables measurable with respect to the finite-point motion  $x_u$ of the Arratia flow.

Among the main instruments in the description of functionals from the Gaussian noise are the It\^o-Wiener expansion and its specification -- the Krylov-Veretennikov expansion \cite{KV}. These notions proved to be fruitfull in the study of filtrations with non-trivial Gaussian parts as well \cite{LeJan_Raimond}.
Consider the Brownian motion $\{w(t), t\geq 0\}.$ Let $L^2(w)$ be the space of square integrable random variables measurable with respect to $\sigma(w).$ The It\^o-Wiener expansion is the representation of $L^2(w)$ as a Hilbert sum of pairwise orthogonal subspaces \cite{Cameron_Martin}
$$
L^2(w)=\oplus^\infty_{n=0} I_n(L^2_{symm}(\mathbb{R}^n_+)),
$$
where $I_n:L^2_{symm}(\mathbb{R}^n_+)\to L^2(w)$ is the operator of $n-$fold stochastic integration with respect to $w,$
$$
I_n a=\int^{\infty}_0 \int^{t_n}_0\ldots \int^{t_2}_0 a(t_1,\ldots ,t_n) dw(t_1) \ldots dw(t_n).
$$

An analogous decomposition holds in a more general case of $L^2(\beta)$, where $\beta$ is a Gaussian random measure on a certain measure space \cite{Ito}.

In the case of Arratia flow, multiple stochastic integrals
\begin{equation}
\label{11}
\int^{\infty}_0\int^{t_n}_0\ldots \int^{t_2}_0 a(t_1,\ldots,t_n) dx(u_1,t_1)\ldots dx(u_n,t_n)
\end{equation}
of different multiplicity are no longer orthogonal  (which in turn comes from the randomness of the quadratic covariation $<x(u,\cdot),x(v,\cdot)>$). This obstacle makes integrals \eqref{11} an inappropriate tool in description of functionals from the Arratia flow. For example, the expansion of a random variable as a series of integrals \eqref{11} is nonunique \cite[Ex. 1]{Dorogovtsev_Riabov}. Analogous problem for the stopped Brownian motion was treated in \cite{DGS} and \cite{Dorogovtsev_Riabov}. Let $\tau$ be the moment when the Brownian motion $w$ has reached a level $a>0.$ Again, multiple stochastic integrals
\begin{equation}
\label{23_02_1}
\int^{\tau}_0 \int^{t_n}_0 \ldots \int^{t_2}_0 a(t_1,\ldots ,t_n) dw(t_1) \ldots dw(t_n)
\end{equation}
of different multiplicity fail to be orthogonal (due to the randomness of the  quadratic variation $<w(\cdot\wedge \tau)>$). An application of the Gram-Schmidt orthogonalization procedure to integrals of the kind \eqref{23_02_1} was studied in \cite{DGS}. However, the structure of integrals \eqref{23_02_1} occured to be too complicated either to find a closed expression for resulting orthogonal objects or to calculate the orthogonal expansion for a specific functional via this procedure. It was noted in \cite{Dorogovtsev_Riabov} that rather to apply classical orthogonalization procedure, a simple transformation of differentials in  \eqref{23_02_1} leads to orthogonal multiple stochastic integrals. Moreover, it was shown that thus constructed integrals constitute an analogue of the It\^o-Wiener expansion for the stopped Brownian motion. Namely, stochastic integrals
\begin{equation}
\label{12}
\begin{gathered}
\int^{\tau}_0 \int^{t_n}_0 \ldots \int^{t_2}_0 a(t_1,\ldots,t_n) (dw(t_1)-f(w(t_1),t_1,t_n)dt_1)\ldots\\
\ldots (dw(t_{n-1})-f(w(t_{n-1}),t_{n-1},t_n)dt_{n-1})dw(t_n),
\end{gathered}
\end{equation}
with $f(y,s,t)=\partial_y \log \mathbb{P}(\forall \ r\in[s,t] \ y+w(r-s)<g(r)),$ corresponding to different $n$ are orthogonal, and any square integrable $w(\cdot\wedge \tau)-$measurable random variable can be represented as a series of integrals \eqref{12}. Thus, the space $L^2(w(\cdot\wedge \tau))$ of all square integrable functionals of $w(\cdot\wedge \tau)$ is the Hilbert sum of pairwise orthogonal subspaces of multiple stochastic integrals \eqref{12}. Moreover, the space of $n-$fold multiple stochastic integrals \eqref{12} is naturally isometric to the space $L^2(\mathbb{P}(\tau>t_n)dt),$ as the squared norm of the integral \eqref{12} equals
$$
\int^{\infty}_0 \int^{t_n}_0 \ldots \int^{t_2}_0 \mathbb{P}(\tau>t_n) a(t_1,\ldots ,t_n)^2dt_1\ldots dt_n.
$$
As opposed to the Gaussian case, stochastic integrals \eqref{12} are no longer polynomials from the stopped Brownian motion.

In the present paper we adopt the approach of \cite{Dorogovtsev_Riabov} to study the structure of the space $L^2(x_u)$ of all square integrable functionals measurable with respect to the $n-$point motion $x_u$ of the Arratia flow. We find suitable transformations of differentials in \eqref{11} that lead to multiple stochastic integrals with the properties

1) multiple stochastic integrals of different multiplicity are orthogonal in $L^2(x_u);$

2) the space of multiple stochastic integrals of a fixed multiplicity is closed in  $L^2(x_u)$ (and in fact is naturally isometric to certain $L^2$ space of integrands);

3) the Hilbert sum of all spaces of multiple stochastic integrals coincides with $L^2(x_u);$

\noindent 
thus defining an analogue of the It\^o-Wiener expansion of the space  $L^2(x_u).$

Next we introduce the notation and briefly describe the construction. In the end of the Introduction we discuss applications of our results to the description of square integrable functionals measurable with respect to the whole Arratia flow $x.$

{\it Notations.}

For any metric space $\mathcal{X},$ $\mathcal{B}(\mathcal{X})$ denotes the Borel $\sigma-$field on $\mathcal{X}.$

Following regions will be used: 

\noindent
$\mathcal{S}^n=\{u\in \mathbb{R}^n: u_1<\ldots <u_n\};$ $\mathcal{S}^n_+=\{u\in \mathcal{S}^n: u_1>0\};$ $\mathcal{S}^{n_1,\ldots,n_m}_+=\mathcal{S}^{n_1}_+\times\ldots\times \mathcal{S}^{n_m}_+.
$

$\mathcal{C}^n$ is the space of all continuous functions $\omega:\mathbb{R}_+\to \mathbb{R}^n$ equipped with the metric of uniform convergence on compacts. $\mathcal{C}^n([0,T])$ is the space of all continuous functions $f:[0,T]\to \mathbb{R}^n$ equipped with the uniform norm.

$\mathcal{B}^n=\mathcal{B}(\mathcal{C}^n)$ is the Borel $\sigma-$field on $\mathcal{C}^n.$

$w:\mathbb{R}_+\times \mathcal{C}^n \to \mathbb{R}^n$ is the canonical process on $\mathcal{C}^n$, i.e. $w(t,\omega)=\omega(t).$

$\mathcal{C}^n$ is equipped with the natural filtration $(\mathcal{B}^n_t)_{t\geq 0},$ i.e. $\mathcal{B}^n_t=\sigma(w(s):0\leq s \leq t).$

$\mathcal{P}$ is the $\sigma-$field of progressively measurable sets on $\mathbb{R}_+\times \mathcal{C}^n.$ It is a $\sigma-$field of all subsets $A\subset \mathbb{R}_+\times \mathcal{C}^n,$ such that for all $t\geq 0$
$$
A\cap [0,t]\times \mathcal{C}^n \in \mathcal{B}([0,t])\times \mathcal{B}^n_t
$$
(see, for example, \cite[Ch. I, \S 4]{Revuz_Yor}). 
Progressively measurable processes arise naturally as integrands in stochastic integrals with respect to the Brownian motion \cite[Ch. IV, \S 2]{Revuz_Yor}. 

Given $u\in \mathbb{R}^n,$ $\mu^u$ denotes the Wiener measure on $(\mathcal{C}^n,\mathcal{B}^n)$ that corresponds to the standard $n-$dimensional Brownian motion starting from $u.$ For every $k\in \{1,\ldots,n\},$  $\mathcal{I}^u_k$ denotes the operator of stochastic integration with respect to the Brownian motion $w_k$ on $(\mathcal{C}^n,\mathcal{B}^n,\mu^u)$:
\begin{equation}
\label{02-1}
\begin{gathered}
\mathcal{I}^u_k:L^2(\mathbb{R}_+\times \mathcal{C}^n,\mathcal{P},dt\times \mu^u(d\omega))\to L^2(\mathcal{C}^n,\mathcal{B}^n,\mu^u), \\
\mathcal{I}^u_k h=\int^\infty_0 h(t)dw_k(t).
\end{gathered}
\end{equation}
According to the Clark representation formula \cite[Ch.5, Th.(3.5)]{Revuz_Yor} each functional $g\in L^2(\mathcal{C}^n,\mathcal{B}^n,\mu^u)$ is uniquely represented as a sum
$$
g=\mathbb{E}^{\mu^u}g + \sum^n_{k=1} \mathcal{I}^u_k g_k.
$$
Denote $\mathcal{Q}^u_k g= g_k,$ so that $\mathcal{Q}^u_k: L^2(\mathcal{C}^n,\mathcal{B}^n,\mu^u) \to L^2(\mathbb{R}_+\times \mathcal{C}^n,\mathcal{P},dt\times \mu^u(d\omega)),$ in such a way that for each  $g\in L^2(\mathcal{C}^n,\mathcal{B}^n,\mu^u),$
\begin{equation}
\label{02-2}
g=\mathbb{E}^{\mu^u}g + \sum^n_{k=1} \mathcal{I}^u_k (\mathcal{Q}^u_k g).
\end{equation}

For all $(k_1,\ldots,k_d)\in \{1,\ldots,n\}^d,$  $I^u_{k_1,\ldots ,k_d}$ denotes  the operator of $d-$fold stochastic integration with respect to Brownian motions $w_{k_1},\ldots ,w_{k_d}$ on $(\mathcal{C}^n,\mathcal{B}^n,\mu^u);$
\begin{equation}
\label{02-3}
\begin{gathered}
I^u_{k_1,\ldots ,k_d}: L^2(\mathcal{S}^d_+)\to L^2(\mathcal{C}^n,\mathcal{B}^n,\mu^u), \\
I^u_{k_1,\ldots ,k_d}a=\int^\infty_0 \ldots \int^{t_2}_0 a(t)dw_{k_1}(t_1)\ldots dw_{k_d}(t_d).
\end{gathered}
\end{equation}
Following formulas hold
$$
\mathbb{E}^{\mu^u} I^u_{k_1,\ldots ,k_d} aI^u_{l_1,\ldots ,l_m} b=\delta_{(k_1,\ldots,k_d),(l_1,\ldots ,l_m)}\int_{\mathbb{S}^d_+} a(t)b(t)dt;
$$
\begin{equation}
\label{13}
\begin{gathered}
I^u_{k_1,\ldots ,k_d} a=\int^\infty_0 \bigg(\int^t_0 \ldots \int^{t_2}_0 a(t_1,\ldots,t_{d-1},t) dw_{k_1}(t_1)\ldots \\
dw_{k_{d-1}}(t_{d-1})\bigg) dw_{k_d}(t)=\int^\infty_0 I^u_{k_1,\ldots ,k_{d-1}} (\tilde{a}(t)) dw_{k_d}(t)=
\\ 
=\mathcal{I}^u_{k_d} (I^u_{k_1,\ldots ,k_{d-1}} (\tilde{a}(\cdot))), \ \mu^u-\mbox{a.e.},
\end{gathered}
\end{equation}
where  
$$
\tilde{a}(t)(t_1,\ldots,t_{d-1})=
\begin{cases} 
a(t_1,\ldots,t_{d-1},t), \ t_{d-1}<t, \\
0, t_{d-1}\geq t.
\end{cases}
$$ 
Existence of a progressively measurable modification for the process \break $\{I^u_{k_1,\ldots ,k_{d-1}}
(\tilde{a}(t)),t\geq 0\}$ follows from the existence of a progressively measurable modification for any measurable adapted process (see, for example, \cite{Ondrejat}).

$\mathcal{K}_n$ is the set of all finite sequences of elements $1,\ldots,n,$ i.e. $\mathcal{K}_n=\cup^\infty_{d=0}\{1,\ldots ,n\}^d;$ the length of a sequence $k\in\mathcal{K}_n$ is denoted by $|k|;$ $\mathcal{K}_{1,\ldots,n}=\mathcal{K}_{1}\times\ldots\times \mathcal{K}_{n}.$ With these notations the It\^o-Wiener expansion for the $n-$dimensional Brownian motion is
\begin{equation}
\label{02_4}
L^2(\mathcal{C}^n,\mathcal{B}^n,\mu^u)=\oplus_{k\in \mathcal{K}_n} I^u_k (L^2(\mathcal{S}^{|k|}_+)).
\end{equation}
Respectively, operators $Q^u_k:L^2(\mathcal{C}^n,\mathcal{B}^n,\mu^u)\to L^2(\mathcal{S}^{|k|}_+)$ are defined in such a way that every $g\in L^2(\mathcal{C}^n,\mathcal{B}^n,\mu^u)$ has a unique series representation
\begin{equation}
\label{02-5}
g=\sum_{k\in \mathcal{K}_n} I^u_k(Q^u_kg ).
\end{equation}

{\it Multiple stochastic integrals with respect to $x_u$.}

Given $u\in\mathcal{S}^n,$ $\nu^u$ denotes the distribution in $(\mathcal{C}^n,\mathcal{B}^n)$ of the $n-$point motion $x_u$  of the Arratia flow started from $u.$  $\tau_{\mathcal{S}^n}$ is the moment when two of trajectories of $x_u$ have met each other. Measures $\nu^u$ and $\mu^u$ coincide on the $\sigma-$field $\mathcal{B}^n_{\tau_{\mathcal{S}^n}}$ \cite{Dorogovtsev_flow}. In particular, when $n=1$, $\nu^u$ is the distribution of the Brownian motion started from $u.$ Hence, the It\^o-Wiener expansion of the space $L^2(\mathcal{C}^1,\mathcal{B}^1,\nu^u)$ is determined by the multiple stochastic integrals along the Brownian motion $w,$ i.e.
$$
L^2(\mathcal{C}^1,\mathcal{B}^1,\nu^u)=\oplus_{k\in \mathcal{K}_1} I^u_k (L^2(\mathcal{S}^{|k|}_+)).
$$

If $n>1,$ the moment $\tau_{\mathcal{S}^n}$ is $\nu^u-$a.s. finite and only two of the trajectories have coalesced, i.e. the set 
$\{w_1(\tau_{\mathcal{S}^n}),\ldots,w_n(\tau_{\mathcal{S}^n})\}$ consists of exactly $n-1$ points. Enumerate these points in the ascending order by $p^{n-1}_1,\ldots ,p^{n-1}_{n-1}$ $(p^{n-1}_1<\ldots<p^{n-1}_{n-1}).$  According to the strong Markov property, the distribution of $x_u$ after the moment $\tau_{\mathcal{S}^n}$ coincides with the distribution of $x_{p^{n-1}}$ (to be precise, it coincides with the distribution of  $x_{p^{n-1}}$ lifted to the space $\mathcal{C}^n$ by repeating a coordinate $j$ for which $w_j(\tau_{\mathcal{S}^n})=w_{j+1}(\tau_{\mathcal{S}^n})=p^{n-1}_j$). Hence, the space $(\mathcal{C}^n,\mathcal{B}^n,\nu^u)$ is naturally isomorphic to the space with mixture of measures $(\mathcal{C}^{n-1}\times\mathcal{C}^{n},\mathcal{B}^{n-1}\times \mathcal{B}^n_{\tau_{\mathcal{S}^n}},\nu^{p^{n-1}(\omega^n)}(d\omega^{n-1})\mu^u(d\omega^n)).$

Consider at first the case $n=2$ and a functional $f\in L^2(\mathcal{C}^2,\mathcal{B}^2,\nu^u).$ To give an example of how our construction works, let us describe the orthogonal expansion of $f$ leaving measurability questions apart. The considerations below are justified in Theorems \ref{stopped}, \ref{finite-dimensional} and Lemmata \ref{measurability1}, \ref{measurability}.

Due to the isomorphism described above, $f=f(\omega^1,\omega^2),$ where $\omega^1$ refers to the only trajectory left after the coalescence, and $\omega^2$ refers to the 2-point motion before the coalescence. For every fixed $\omega^2,$  $f(\cdot,\omega^2)\in L^2(\mathcal{C}^1,\mathcal{B}^1,\mu^{p^1(\omega^2)}).$ Hence, $f(\cdot,\omega^2)$ possesses an It\^o-Wiener expansion
$$
f(\cdot,\omega^2)=\sum_{k^1\in\mathcal{K}_1}I^{p^1(\omega^2)}_{k^1}(a_{k^1}(\cdot,\omega^2)).
$$
In this expansion, each kernel $a_{k^1}$ is a functional of $\omega^2$ and in fact
$$
a_{k^1}(t^1,\cdot)\in L^2(\mathcal{C}^{2},\mathcal{B}^2_{\tau_{\mathcal{S}^2}},\mu^u).
$$
By the Theorem \ref{stopped}, $a_{k^1}(t^1,\cdot)$ can be further expanded as a sum of multiple stochastic integrals with respect to the stopped Brownian motion \eqref{281}
$$
a_{k^1}(t^1,\cdot)=\sum_{k^2\in\mathcal{K}_2}\mathcal{J}^{u,\mathcal{S}^2}_{k^2}(a_{k^1,k^2}(t^1,\cdot)).
$$
Finally, in $L^2(\mathcal{C}^2,\mathcal{B}^2,\nu^u)$
\begin{equation}
\label{intro1}
f=\sum_{(k^1,k^2)\in\mathcal{K}_{1,2}}\mathcal{A}^u_{k^1,k^2}a_{k^1,k^2},
\end{equation}
where
\begin{equation}
\label{intro11}
\mathcal{A}^u_{k^1,k^2}a_{k^1,k^2}(\omega^1,\omega^2)=I^{p^1(\omega^2)}_{k^1}(\mathcal{J}^{u,\mathcal{S}^2}_{k^2}(a_{k^1,k^2}(t^1,\cdot))(\omega^2))(\omega^1),
\end{equation}
all the summands are pairwise orthogonal. Theorem \ref{stopped} implies that the squared norm of $I^{p^1}_{k^1}(\mathcal{J}^{u,\mathcal{S}^2}_{k^2}a_{k^1,k^2})$ equals 
\begin{equation}
\label{intro2}
\int_{\mathcal{S}^{|k^1|,|k^2|}_+}a_{k^1,k^2}(t^1,t^2)^2\alpha_{\mathcal{S}^{2}}(u,t^2_{|k^2|})dt^1dt^2,
\end{equation}
 $\alpha_{\mathcal{S}^2}(u,t)$ is the probability that trajectories of a two-dimensional Brownian motion started from $u$ haven't met up to the moment $t.$
$\mathcal{A}^u_{k^1,k^2}$ is the operator of multiple stochastic integration with respect to $x_u.$ The index $k^1$ defines the multiplicity of integrals along the only trajectory left after the  coalescence and the index $k^2$ defines the order of differentials and the  multiplicity of integrals along trajectories before the coalescence. Respectively, the kernel $a_{k^1,k^2}$ is a function of $(t^1,t^2),$ where $t^j=(t^j_1,\ldots,t^j_{|k^j|})$ varies over the $|k^j|-$dimensional simplex $\mathcal{S}^{|k^j|}_+.$ To handle questions like measurability  of the expression in  \eqref{intro11}, we give following definition. Let $(\mathcal{X},\mathcal{B}),$ $(\Omega,\mathcal{F})$ be measurable spaces, $(\mu^\omega)_{\omega\in \Omega}$ be a regular family of measures on $(\mathcal{X},\mathcal{B}),$ that is a mapping $(B,\omega)\to \mu^\omega(B)$ of $\mathcal{B}\times \Omega$ into $[0,\infty],$ such that

1) for all $\omega\in \Omega$ $\mu^\omega$ is a (possibly infinite) measure on $(\mathcal{X},\mathcal{B});$

2) for every $B\in \mathcal{B}$ the mapping $\omega\to \mu^\omega(B)$ is $\mathcal{F}-$measurable.

\noindent
For a detailed exposition of the theory of regular measures see \cite[Ch. 10]{Bog_measure}.

\begin{dfn}
\label{measurable_family} Assume that for each $\omega\in \Omega,$
$\xi^{\omega}$ is a measurable function on $(\mathcal{X},\mathcal{B}).$ If there exists measurable mapping $h:\mathcal{X}\times\Omega\to \mathbb{R}$ such that
$$
\forall \omega\in \Omega \ h(\cdot,\omega)=\xi^\omega,  \ \mu^\omega-\mbox{a.s.,}
$$
then we will say that a family $\{\xi^\omega\}_{\omega\in \Omega}$ can be realized as a measurable function on $\mathcal{X}\times\Omega$ with respect to the family $\{\mu^\omega\}_{\omega\in \Omega}.$ When all measures $\mu^\omega$ are equal to some measure $\mu$ we will say that a family $\{\xi^\omega\}_{\omega\in \Omega}$ can be realized as a measurable function with respect to the measure $\mu.$
\end{dfn}

\noindent
Section 4 is devoted to two general results on measurable realizations which cover all the measurability issues of our construction.

In section 3 the described approach will be carried out to define multiple stochastic integrals with respect to $x_u$ that produce an analogue of the It\^o-Wiener expansion of the space $L^2(\mathcal{C}^n,\mathcal{B}^n,\nu^u)$ (Theorem \ref{finite-dimensional}). The only modification comes from the formula  \eqref{intro2}. Consider, for example, $n=3.$ Then every $f\in L^2(\mathcal{C}^3,\mathcal{B}^3,\nu^u)$ is a functional of $(\omega^2,\omega^3),$ where  $\omega^2$ refers to two trajectories left after the first coalescence, and $\omega^3$ refers to the 3-point motion before the first coalescence. As above, $f(\cdot,\omega^3)$ is expanded into the sum
$$
f(\cdot,\omega^3)=\sum_{(k^1,k^2)\in\mathcal{K}_1\times \mathcal{K}_2}\mathcal{A}^{p^2(\omega^3)}_{k^1,k^2}(a_{k^1,k^2}(\cdot,\omega^3)).
$$
Due to \eqref{intro2} the squared norm of the summand $\mathcal{A}^{p^2}_{k^1,k^2}a_{k^1,k^2}$ equals
\begin{equation}
\label{intro3}
\int_{\mathcal{C}^3}\int_{\mathcal{S}^{|k^1|,|k^2|}_+}a_{k^1,k^2}(t^1,t^2,\omega^3)^2\alpha_{\mathcal{S}^{2}}(p^2(\omega^3),t^2_{|k^2|})dt^1dt^2\mu^u(d\omega^3),
\end{equation}
what means that $a_{k^1,k^2}(t^1,t^2,\cdot)\in L^2(\mathcal{C}^3,\mathcal{B}^3_{\tau_{\mathcal{S}^3}},\alpha_{\mathcal{S}^{2}}(p^2(\omega^3),t^2_{|k^2|})\mu^u(d\omega^3)).$ \break Hence, for $a_{k^1,k^2}(t^1,t^2,\cdot)$ to become a square integrable functional from the stopped Brownian motion, an additional weight $\alpha_{\mathcal{S}^{2}}(p^2(\omega^3),t^2_{|k^2|})$ is needed. The problem of expanding such functionals as a series of pairwise orthogonal multiple stochastic integrals is solved in the Lemma \ref{stopped_absolutely_continuous} for a class of measures $\varkappa\ll \mu^u$ on $(\mathcal{C}^n,\mathcal{B}^n_{\tau}).$ Obtained results are used in the inductive definition \eqref{operators} of multiple stochastic integrals with respect to $n-$point motion of the Arratia flow. 

In Theorem \ref{finite-dimensional} it is proved that each $f\in L^2(\mathcal{C}^n,\mathcal{B}^n,\nu^u)$ is uniquely represented as a series of pairwise orthogonal multiple stochastic integrals
$$
f=\sum_{(k^1,\ldots,k^n)\in\mathcal{K}_{1,\ldots,n}} \mathcal{A}^u_{k^1,\ldots,k^n}a_{k^1,\ldots,k^n}.
$$
$\mathcal{A}^u_{k^1,\ldots,k^n}$ is an operator of multiple stochastic integration with respect to the trajectories of $x(u,\cdot).$ It contains $|k^j|$ integrals over a region where exactly $j$ particles, integrals are taken with respect to the trajectories of these particles transformed in the manner of \eqref{12}. Differentials are transformed with the help of mappings \eqref{25_02_1} and \eqref{31_1}. The squared norm of the summand $\mathcal{A}^u_{k^1,\ldots,k^n}a_{k^1,\ldots,k^n}$ equals
$$
\int_{\mathcal{S}^{|k^1|,\ldots,|k^n|}_+}a_{k^1,\ldots,k^n}(t^1,\ldots,t^n)^2\rho_{t^2_{|k^2|},\ldots,t^n_{|k^n|}}(u)dt^1\ldots dt^n,
$$
where functions $\rho_{t^2,\ldots,t^n}$ are defined in \eqref{rho_eq}.

{\it Multiple stochastic integrals with respect to $x.$}

Obtained results indicate possible way to define multiple stochastic integrals with respect to the whole Arratia flow. Recall that $\sigma_n$ is the moment when all particles $\{x(u,\cdot),u\in [0,1]\}$ have coalesced into exactly $n$ particles. Denote $p^n_1<\ldots<p^n_n$ the positions of $n$ remaining particles at the moment $\sigma_n.$ The strong Markov property of $x$ imply that $L^2(\Omega,\mathcal{F}^x,\mathbb{P})$ is naturally isometric to $L^2(\mathcal{C}^{n}\times\Omega,\mathcal{B}^{n}\times \mathcal{F}^x_{\sigma_n},\nu^{p^{n}(\omega_n)}(d\omega^{n})\mathbb{P}(d\omega_n)),$ where $\omega^n$ refers to $n$ trajectories left after the moment $\sigma_n$ and $\omega_n$ refers to the flow $x$ before the moment $\sigma_n.$ From the
Theorem \ref{finite-dimensional} it follows that each $\alpha\in L^2(\Omega,\mathcal{F}^x,\mathbb{P})$ has a series expansion
\begin{equation}
\label{24_1}
\alpha=\sum_{(k_1,\ldots,k_n)\in\mathcal{K}_{1,\ldots,n}} \mathcal{A}^{p^n(\omega_n)}_{k_1,\ldots,k_n}(a_{k_1,\ldots,k_n}(\cdot,\omega_n)).
\end{equation}
Denote $\mathcal{K}_{\mathbb{N}}=\prod_{n\geq 1} \mathcal{K}_n.$ Due to the orthogonality of summands in \eqref{24_1}, for each $k\in \mathcal{K}_\mathbb{N},$ $k=(k_1,k_2,\ldots),$
elements $\mathcal{A}^{p^n(\omega_n)}_{k_1,\ldots,k_n}(a_{k_1,\ldots,k_n}(\cdot,\omega_n))$ converge to some element $P_k \alpha$ in $L^2(\Omega,\mathcal{F}^x,\mathbb{P}).$ $(P_k \alpha)_{k\in {\mathcal{K}_\mathbb{N}}}$ is a continuum family of pairwise orthogonal elements, which may be considered as multiple stochastic integrals of some kernels with respect to $x.$ If such point of view is possible, obtained multiple stochastic integrals would be natural candidates to form an analogue of the It\^o-Wiener expansion of $L^2(\Omega,\mathcal{F}^x,\mathbb{P}).$ As for now, the description of $P_k\alpha$ and construction of the expansion of $L^2(\Omega,\mathcal{F}^x,\mathbb{P})$ are open problems. We leave them for further investigations.

\section{Stopped Brownian Motion}

Let $G\subset \mathbb{R}^n$ be an open connected set and $\aleph:G\to \mathbb{R}^n$ be an infinitely differentiable vector field on $G.$ Given $\omega\in \mathcal{C}^n,$ consider the following integral equation
\begin{equation}
\label{30_1}
\xi(t)=\omega(t)+\int^t_0 \aleph(\xi(s))ds, \ t\geq 0.
\end{equation}
Of course, the solution to \eqref{30_1} may not exist for all $t>0.$ The precise definition of the solution and its properties are given below. Though the result seems known we add the proof because our situation differs from the usual one - we seek for a solution that is an adapted functional on the space $(\mathcal{C}^n,\mathcal{B}^n)$ without referring to any probability measure. 

Consider the set 
$$
\mathcal{D}=\{(T,\omega)\in \mathbb{R}_+\times \mathcal{C}^n: \ \mbox{there exists continuous function} \ \xi:[0,T]\to G, 
$$
$$
\mbox{such that for all} \ t\in [0,T] \ \mbox{\eqref{30_1} holds}\}.
$$

\begin{lem}
\label{add4}
1) Each section $\mathcal{D}_\omega$ is the interval of the form $[0,\tau(\omega))$  $(\mathcal{D}_\omega=\emptyset,$ if $\tau(\omega)=0$ and $\mathcal{D}_\omega=\mathbb{R}_+,$ if $\tau(\omega)=\infty).$  In particular, $\mathcal{D}=\{(T,\omega)\in \mathbb{R}_+\times \mathcal{C}^n:\tau(\omega)>T\}.$

2) For each $\omega\in \mathcal{C}^n$ there exists unique continuous function $\xi(\cdot,\omega):[0,\tau(\omega)) \to G,$ such that for all $t\in [0,\tau(\omega))$ \eqref{30_1} holds.

3) $\mathcal{D}\in \mathcal{P},$ i.e. $\mathcal{D}$ is a progressively measurable subset of $\mathbb{R}_+\times \mathcal{C}^n.$

4) $\tau:\mathcal{C}^n\to [0,\infty]$  is a stopping time.

5) $\xi:\mathcal{D}\to G$ is a progressively measurable process.

6) Relatively to Wiener measures $\mu^u, \ u\in G,$ $\xi$ is a strong Markov process \cite[Ch. III, \S 3]{Dynkin}, i.e. for any $\mathcal{B}^n-$stopping time $\sigma\leq \tau,$ $t\geq 0$ and a Borel set $A\subset G,$
$$
\mu^u(\xi(t+\sigma)\in A, \ t+\sigma<\tau|\mathcal{B}^n_{\sigma})=1_{\sigma<\tau}\mu^{v} (\xi(t)\in A, \ t<\tau)|_{v=\xi(\sigma)}.  
$$
\end{lem}

\begin{proof}
1) and 2). Assume that $T\in \mathcal{D}_\omega.$ Evidently,  $[0,T]\subset \mathcal{D}_\omega.$ There exists continuous function $\xi:[0,T]\to G,$ such that for all $t\in [0,T]$ \eqref{30_1} holds. We will show that for $\delta$ small enough, $\xi$ can be uniquely extended to continuous function on $[0,T+\delta]$ that satisfies \eqref{30_1} for all $t\in [0,T+\delta].$ To this end it is enough to prove that there exists continuous function $\xi_1:[0,\delta]\to G$ that satisfies
$$
\xi_1(t)=\omega_1(t)+\int^t_0 \aleph(\xi_1(s))ds, \ t\in[0,\delta],
$$
with $\omega_1(\cdot)=\xi(T)+\omega(\cdot+T)-\omega(T).$ The difference $\xi_2(t)=\xi_1(t)-\omega_1(t)$ must then be a solution to the Cauchy problem 
\begin{equation}
\label{08_04_1}
\begin{cases}
d\xi_2(t)=\aleph(\xi_2(t)+\omega_1(t))dt, \\
\xi_2(0)=0.
\end{cases}
\end{equation}
Denote $B(\delta)$ a closed ball in $\mathbb{R}^n$ with centre $0$ and radius $\delta.$ For small enough $\delta,$ a compact set 
$K=\{x+\omega_1(t): \ t\in [0,\delta], x\in B(\delta)\}$ is a subset of $G.$
Hence, $C=\sup_{K} |\nabla\aleph|<\infty,$ and for all $(t,x_1),(t,x_2)\in [0,\delta]\times B(\delta)$ one has
$$
|\aleph(x_1+\omega_1(t))-\aleph(x_2+\omega_1(t))|\leq C|x_1-x_2|.
$$
The existence and uniqueness of a solution to \eqref{08_04_1} now follows from \cite[Ch. 1, Th. 2.3]{CL}. 

3) Given $S\geq 0$ let us show that $\mathcal{D}\cap ( [0,S]\times \mathcal{C}^n)\in \mathcal{B}([0,S])\times \mathcal{B}^n_S.$ This is exactly what the progressive measurability means. Denote $G^{\delta}=\{x\in G: \mbox{dist}(x,\partial G)\geq \delta\}.$ For all $T\geq 0$ and $\delta>0$ the space $\mathcal{C}^n([0,T];G^\delta)$ of all continuous functions $f:[0,T]\to G^\delta$ is a closed subset of $\mathcal{C}^n([0,T]).$ In particular, it is a Polish space. The mapping $F^{T,\delta}:\mathcal{C}^n([0,T];G^\delta)\to \mathcal{C}^n([0,T]),$
$$
F^{T,\delta}(f)(t)=f(t)-\int^t_0 \aleph(f(s))ds, \ t\in [0,T],
$$
is injective (see part 1) of the proof). By the Souslin Theorem \cite[Th. 6.8.6]{Bog_measure}, the image $F^{T,\delta}(\mathcal{C}^n([0,T];G^\delta))$ is a Borel set in $\mathcal{C}^n([0,T]).$ Let $A_S$ be some dense countable set in $[0,S],$ such that $S\in A_S.$ Then 
$$
\mathcal{D}\cap ( [0,S]\times \mathcal{C}^n)=
\bigcup_{\delta>0} \bigcup_{T\in A_S} [0,T]\times F^{T,\delta}(\mathcal{C}^n([0,T];G^\delta)) \in \mathcal{B}([0,S])\times \mathcal{B}^n_S.
$$
It is enough to take union only in $T\in A_S,$ as every solution to \eqref{30_1} can be continued from a closed interval to an interval with the right end  in $A_S$ (see part 1) of the proof).

4) Follows from 3):
$$
\{\omega:\tau(\omega)>t\}=\mathcal{D}_t\in \mathcal{B}^n_t.
$$

5) By continuity of $\xi$ and \cite[Ch. 1, Prop. (4.8)]{Revuz_Yor}, it is enough to prove that $\xi(T)$ is  $\mathcal{B}^n_T-$measurable on the set $\{\tau>T\}=\mathcal{D}_T.$ Recall mappings $F^{T,\delta}$ from part 3) of the proof. Then, for every Borel set $\Delta\subset G,$
$$
\{\omega \in \mathcal{D}_T: \xi(T,\omega)\in \Delta\}=
$$
$$
=\bigcup_{\delta>0} \{\omega \in \mathcal{D}_T: \xi(T,\omega)\in \Delta \ \mbox{and} \ \xi([0,T],\omega)\subset G^\delta \}=
$$
$$
=\bigcup_{\delta>0} \{\omega \in \mathcal{D}_T: \ \omega|_{[0,T]} \in F^{T,\delta}(\{f\in \mathcal{C}^n([0,T];G^\delta): f(T)\in \Delta\})\}.
$$ 
By the Souslin Theorem, the latter set belongs to $\mathcal{B}^n_T.$

6) Introduce shift operators $\theta_{r,v}\omega=v+\omega(\cdot+r)-\omega(r).$ If $\tau(\omega)>s,$ then $\tau(\omega)>t+s$ if and only if $\tau(\theta_{s,\xi(s,\omega)})>t,$ and in this case
$$
\xi(s+\cdot,\omega)=\xi(\cdot,\theta_{s,\xi(s,\omega)}\omega) \ \mbox{on} \ [0,t].
$$ 
Hence, by the strong Markov property of the Wiener process \cite[Ch. 3, Cor. (3.6)]{Revuz_Yor},
$$
\mu^u(\xi(t+\sigma)\in \Delta,\tau>t+\sigma/\mathcal{B}^n_\sigma)=
$$
$$
=1_{\tau>\sigma}\mu^u(\xi(t,\theta_{\sigma,\xi(\sigma)}w) \in \Delta, \tau(\theta_{\sigma,\xi(\sigma)}w)>t /\mathcal{B}^n_\sigma)=
$$
$$
=1_{\tau>\sigma}\mu^v(\xi(t) \in \Delta, \tau>t)|_{v=\xi(\sigma)}.
$$
\end{proof}
$\tau$ will be referred to as the lifetime of the solution to \eqref{30_1}. The main result of this section is an analogue of the It\^o-Wiener expansion of the space $L^2(\mathcal{C}^n,\mathcal{B}^n_{\tau},\mu^u).$

Introduce the function 
\begin{equation}
\label{25_02_3}
\alpha(t,u)=\mu^u(\tau>t), \ u\in G, \ t\in\mathbb{R}. 
\end{equation}
In the section 3 we will write $\alpha_{\aleph,G}$ instead of $\alpha$ to indicate its  dependence on  the domain $G$ and the vector field $\aleph.$ Note, that $\alpha(t,u)=1$ for $t\leq 0.$

\begin{lem}
\label{probability}
1)  $\alpha$  satisfies the equation
\begin{equation}
\label{21}
\frac{\partial \alpha}{\partial t}(t,u) =\frac{1}{2} \Delta_u \alpha(t,u)+ (\aleph(u),\nabla_u \alpha(t,u)), t> 0 , \ u\in G.
\end{equation}

In particular, $\alpha$ is infinitely differentiable in $\mathbb{R}_+\times G.$

2) The Clark representation formula \cite[Ch.5, Th.(3.5)]{Revuz_Yor} for $1_{\tau>t}$ is
\begin{equation}
\label{22}
1_{\tau>t}=\alpha(t,u)+\int^{t\wedge \tau}_0 \nabla_u \alpha(t-s,\xi(s))dw(s), \ \mu^u-\mbox{a.s.}
\end{equation}

3) $\mathbb{E}^{\mu^u}[1_{\tau>t}/\mathcal{B}^n_s]=1_{\tau>t\wedge s}\alpha(t-t\wedge s,\xi(t\wedge s)), \ \mu^u-\mbox{a.s.}$
\end{lem}

\begin{proof}
1) Consider an open ball $B\Subset G$ and a starting point $u\in B.$ Let $\tau_B$ be the moment when the process $\xi$ leaves $B.$ Denote $g_B(t,u,v)$ the distribution density of the pair $(\tau_B,\xi(\tau_B))$ with respect to $dt\times \sigma(dv),$ where $\sigma$ is the surface measure on $B.$ Such density exists since for all $t>0$ the distribution of $\xi(\cdot\wedge \tau_B\wedge t)$ is equivalent to the distribution of the stopped Brownian motion $w(\cdot \wedge \tau^w_B\wedge t),$ $\tau^w_B$ is the moment when $w$ leaves $B$  \cite[Th. 7. 10]{Liptser_Shiryaev}, while the distribution density of $(\tau^w_B,w(\tau^w_B))$ is known explicitly \cite{Hsu}. From the strong Markov property of $\xi$ following representation follows
\begin{equation}
\label{23}
\alpha(t,u)=\int^t_{-\infty}\int_{\partial B}\alpha(s,v)g_B(t-s,u,v)\sigma(dv)ds, \ t>0, u\in B.
\end{equation}
If one knows that $g_B(\cdot,v)$ satisfies \eqref{21} in $\mathbb{R}_+\times B$, then the representation \eqref{23} implies that  $\alpha$ solves \eqref{21} as a distribution on $\mathbb{R}_+\times B.$ Next, according to the hypoellipticity of the operator $\frac{1}{2} \Delta_u + (\aleph(u),\nabla_u)-\frac{\partial}{\partial t}$ \cite[Th. 3.4.1]{Stroock_PDE},  $\alpha$ is infinitely differentiable on $\mathbb{R}_+\times B$ and \eqref{21} holds in the usual sense.

Let us check \eqref{21} for $g_B.$ Denote $p_B(t,x,y)$ the transition density of the process $\xi$ killed at the moment $\tau_B$ \cite[\S 5.2]{Stroock_PDE}.  Due to the Markov property of $\xi$, 
$$
\int^\infty_t \int_A g_B(s,u,v)\sigma(dv)ds=\mu^u(\tau_B>t, \xi(\tau_B)\in A)=
$$
$$
=\int_B p_B(t,u,y) \int^\infty_0 \int_A g_B(s,y,v)\sigma(dv) ds dy.
$$
Which implies the representation
$$
g_B(t,u,v)=-\int_{B} \frac{\partial p_B}{\partial t} (t,u,y) \int^\infty_0 g_B(s,y,v)ds dy.
$$
$p_B(\cdot,y)$ satisfies \eqref{21} \cite[Th. 5.2.8]{Stroock_PDE}. Hence, another application of the hypoellipticity of $\frac{1}{2} \Delta_u + (\aleph(u),\nabla_u)-\frac{\partial}{\partial t}$ proves \eqref{21} for $g_B.$

2) At first we consider the case when $G$ is bounded with an infinitely smooth boundary and $\aleph$ is a restriction to $G$ of an infinitely differentiable compactly supported vector field on $\mathbb{R}^n.$ In this case there exist global solution $\xi$ of \eqref{30_1}. We will understand $\tau$ as the moment when $\xi$ leaves $G.$ In this case $\alpha(t,u)=\int_G p_G(t,u,y) dy,$ where $p_G$ is a transition density of the process $\xi$ killed at the moment $\tau$ \cite[\S 5.2]{Stroock_PDE}. $p_G$ is a Green function of the parabolic boundary value problem
$$
\begin{cases}
\frac{\partial f}{\partial t}(t,u) =\frac{1}{2} \Delta_u f(t,u)+ (\aleph(u),\nabla_u f(t,u)), \ u\in G, t> 0, \\
f(t,u)=0, \ u\in \partial G, t>0.
\end{cases}
$$
The explicit construction of the Green function $p_G$ \cite[\S VI.2.1]{Garroni_Menaldi} leads to
$$
\sup_{t\geq \varepsilon} \bigg(\alpha(t,u)+\bigg|\frac{\partial \alpha}{\partial t}(t,u)\bigg|\bigg)\to 0, \ u\to u_0\in\partial G.
$$

Recall that $G^{\delta}$ is the set of points of $G$ whose distance to $\partial G$ exceeds $\delta.$ Denote $\tau_\delta=\inf\{t\geq 0: \xi(t)\not\in G^\delta\}.$  From the It\^o's formula applied to the function $\alpha(t-s, \xi (s\wedge \tau_\delta)),$ $0\leq s\leq t-\varepsilon$ it follows that $\mu^u-$a.s.
\begin{equation}
\label{24}
\begin{gathered}
\alpha(\varepsilon, \xi((t-\varepsilon)\wedge \tau_\delta))=\alpha(t,u)+\int^{(t-\varepsilon)\wedge\tau_\delta}_0 \nabla_u\alpha(t-s,\xi(s))dw(s) \\
-1_{\tau_\delta<t-\varepsilon}\int^{t-\varepsilon}_{\tau_\delta}
\frac{\partial \alpha}{\partial t}(t-s, \xi(\tau_\delta))ds.
\end{gathered}
\end{equation}

\noindent
Consider the case $\tau>t.$ Then for small enough $\varepsilon$ and $\delta$ one has $\tau_\delta>t-\varepsilon$ and the left-hand side of \eqref{24} transforms into $\alpha(\varepsilon, \xi(t-\varepsilon)).$ $\xi(t-\varepsilon)\to \xi(t)\in G,$ so, by continuity of $\xi,$
$$
\alpha(\varepsilon, \xi(t-\varepsilon))\to \alpha(0,\xi(t))=1, \ \varepsilon\to 0.
$$
Hence, multiplying \eqref{24} by $1_{\tau>t}$ and taking limits $\varepsilon,\delta\to 0$ implies
the relation
\begin{equation}
\label{25}
1_{\tau>t}=1_{\tau>t} \alpha(t,u) +1_{\tau>t} \int^{t}_0 \nabla_u\alpha(t-s,\xi(s))dw(s).
\end{equation}
Consider the case $\tau<t-\varepsilon.$ Then for all $\delta,$ $\tau_\delta<t-\varepsilon$ and the left-hand side of \eqref{24} is 
$\alpha(\varepsilon, \xi(\tau_\delta)).$ $\xi(\tau_\delta)\to \xi(\tau)\in \partial G$ and, from the boundary conditions, one has 
$$
\alpha(\varepsilon, \xi(\tau_\delta))\to \alpha(\varepsilon, \xi(\tau))=0, \ \delta\to 0.
$$
Hence, multiplying \eqref{24} by $1_{\tau<t-\varepsilon}$ and taking the limit $\delta\to 0$ one gets
\begin{equation}
\label{26}
0=1_{\tau<t-\varepsilon} \alpha(t,u) + 1_{\tau<t-\varepsilon} \int^{\tau}_0 \nabla_u\alpha(t-s,\xi(s))dw(s).
\end{equation}
Finally, \eqref{25} and \eqref{26} imply the relation \eqref{22}.

According to \cite[Ch. 5, Th. 4.20]{Edmunds_Evans}, $G$ can be written as an increasing union of bounded domains  $G_n\Subset G$ with infinitely differentiable  boundaries. Denote $\tau_n=\inf\{t\geq 0: \xi(t)\not\in G_n\},$ $\alpha_n(t,u)=\mu^u(\tau_n>t).$ As $\tau_n\nearrow \tau,$ it follows that functions $\alpha_n$ converge pointwise to $\alpha.$ From the representation \eqref{23} it follows that all derivatives of $\alpha_n$ also converge to respective derivatives of $\alpha.$ Hence, \eqref{22} holds for any $G$ and $\aleph.$

The property 3) is an immediate consequence of the Markov property of $\xi.$

\end{proof}

Denote $\pi^{t,u}$ the probability on $(\mathcal{C}^n,\mathcal{B}^n)$ defined via the density
$$
\frac{d\pi^{t,u}}{d\mu^u}=\alpha(t,u)^{-1}1_{\tau>t}.
$$
In what follows we introduce a family of transformations of $\mathcal{C}^n$ that send $\pi^{t,u}$ into $\mu^u.$
Recall the equality $\mathcal{D}=\{(t,\omega):\tau(w)>t\}$ (Lemma \ref{add4}). Define mappings
\begin{equation}
\label{25_02_1}
\begin{gathered}
\phi(t,\omega)=1_{\tau(\omega)>t} (\omega(\cdot)-\int^{\cdot\wedge t}_0 \nabla_u \log \alpha(t-s,\xi(s,\omega))ds), \\
\Phi(t,\omega)=(t,\phi(t,\omega))
\end{gathered}
\end{equation}
and the set  $\mathcal{R}=\Phi(\mathcal{D}).$ Next Lemma states measurability properties of $\Phi$ with respect to the Borel $\sigma-$field $\mathcal{B}(\mathbb{R}_+\times \mathcal{C}^n)$ and the $\sigma-$field $\mathcal{P}$ of progressively measurable sets on $\mathbb{R}_+\times \mathcal{C}^n$ \cite[Ch. I, \S 4]{Revuz_Yor}.

\begin{lem}
\label{flow}
1) $\Phi:\mathbb{R}_+\times \mathcal{C}^n \to \mathbb{R}_+\times \mathcal{C}^n$ is both $\mathcal{B}(\mathbb{R}_+\times \mathcal{C}^n)/\mathcal{B}(\mathbb{R}_+\times \mathcal{C}^n)-$ and $\mathcal{P}/\mathcal{P}-$measurable;

2) $\mathcal{D},\mathcal{R}\in \mathcal{P}$ and $\Phi$ is a bimeasurable bijection of $\mathcal{D}$ onto $\mathcal{R}$ when both sets are simultaneously equipped either with the Borel or with the  progressively measurable $\sigma-$fields. In particular, there exists $\mathcal{B}(\mathbb{R}_+\times \mathcal{C}^n)/\mathcal{B}(\mathbb{R}_+\times \mathcal{C}^n)-$ and $\mathcal{P}/\mathcal{P}-$measurable mapping $\Psi:\mathbb{R}_+\times \mathcal{C}^n\to \mathbb{R}_+\times \mathcal{C}^n$ that coincides with $\Phi^{-1}$ on $\mathcal{R};$

3) For each $t\geq 0$
$$
\pi^{t,u}\circ \phi(t,\cdot)^{-1}=\mu^u;
$$

\end{lem}

\begin{proof}

1) To prove $\mathcal{P}/\mathcal{P}-$measurability of $\Phi$ it is enough to check that the restriction of $\Phi$ onto $[0,T]\times \mathcal{C}^n$ is $\mathcal{B}([0,T])\times \mathcal{B}^n_T/\mathcal{P}-$measurable. As $\Phi([0,T]\times \mathcal{C}^n)\subset [0,T]\times \mathcal{C}^n,$ the needed measurability is equivalent to $\mathcal{B}([0,T])\times \mathcal{B}^n_T/\mathcal{B}([0,T])\times \mathcal{B}^n_T-$measurability, which immediately follows from the definition.

2) The inclusion $\mathcal{D}\in\mathcal{P}$ was proved in the Lemma \ref{add4}. Let us check that $\Phi$ is injective on $\mathcal{D}.$ Assume that $\Phi(t,\omega)=\Phi(t',\omega'),$ 
$\tau(\omega)>t,$ $\tau(\omega')>t'.$ Then $t=t',$ $\varphi(t,\omega)=\varphi(t,\omega').$ The last equality means that for all $r\geq 0$
\begin{equation}
\label{rev1}
\omega(r)-\int^{r\wedge t}_0 \nabla \log \alpha(t-s,\xi(s,\omega))ds=\omega'(r)-\int^{r\wedge t}_0 \nabla \log \alpha(t-s,\xi(s,\omega'))ds.
\end{equation}
In particular, $\omega(0)=\omega'(0)$ and $\xi(0,\omega)=\xi(0,\omega').$ After the moment $t$ the difference between $\omega$ and $\omega'$ is constant. Consider differences $\omega(r)-\omega'(r)$ and $\xi(r,\omega)-\xi(r,\omega'),$ $r\in [0,t].$ They are differentiable on $[0,t)$, as it follows from \eqref{rev1} and \eqref{30_1}, and 
$$
\frac{\partial (\omega-\omega')}{\partial r}=\nabla\log \alpha(t-r,\xi(r,\omega)) -\nabla\log \alpha(t-r,\xi(r,\omega')),
$$
$$
\frac{\partial (\xi(\cdot,\omega)-\xi(\cdot,\omega'))}{\partial r}=\aleph(\xi(r,\omega))-\aleph(\xi(r,\omega')) +\frac{\partial (\omega-\omega')}{\partial r}.
$$
Combining the relations one has
$$
\frac{\partial (\xi(\cdot,\omega)-\xi(\cdot,\omega'))}{\partial r}=\aleph(\xi(r,\omega))-\aleph(\xi(r,\omega')) +
$$
$$
+\nabla\log \alpha(t-r,\xi(r,\omega)) -\nabla\log \alpha(t-r,\xi(r,\omega')).
$$
In the view of the infinite differentiability of $\alpha$ (Lemma \ref{probability}) and $\aleph$, for all $r\leq t-\varepsilon$ 
$$
\bigg|\frac{\partial (\xi(\cdot,\omega)-\xi(\cdot,\omega'))}{\partial r}\bigg|\leq C |\xi(r,\omega)-\xi(r,\omega')| 
$$
By the Gronwall lemma, the difference  
$\xi(r,\omega)-\xi(r,\omega')$ is bounded by a multiple of $\xi(0,\omega)-\xi(0,\omega')=0.$ Hence, $\xi(r,\omega)=\xi(r,\omega')$ and $\omega=\omega'.$

It follows that $\Phi:\mathcal{D}\to \mathcal{R}$ is a bijection. The Souslin Theorem \cite[Th. 6.8.6]{Bog_measure} implies that the image under $\Phi$ of every Borel subset of $\mathcal{D}$ is a Borel subset of $\mathcal{R}.$ It remains to check that for every progressively measurable set $A\subset\mathcal{D}$ its image $\Phi(A)$ is progressively measurable. Given $T\geq 0$ denote $\Phi^T$ the restriction of $\Phi$ onto $[0,T]\times \mathcal{C}^n.$ The $\sigma-$field $\mathcal{B}([0,T])\times \mathcal{B}^n_T$ is naturally isomorphic to the Borel $\sigma-$filed $\mathcal{B}([0,T]\times \mathcal{C}^n([0,T])).$ Another application of the Souslin theorem gives the inclusion $\Phi(A)\bigcap ([0,T]\times \mathcal{C}^n)=\Phi^T(A\bigcap ([0,T]\times \mathcal{C}^n))\in\mathcal{B}([0,T])\times \mathcal{B}^n_T.$

3) The Clark representation for the density $\rho=\frac{d\pi^{t,u}}{d\mu^u}=\alpha(t,u)^{-1}1_{\tau>t}$ was derived in the Lemma \ref{probability}, 2): 
$$
\rho=1+\alpha(t,u)^{-1}\int^{t\wedge \tau}_0 \nabla_u \alpha(t-s,\xi(s))dw(s), \ \mu^u-\mbox{a.s.}
$$
Its conditional expectation with respect to the $\mathcal{B}^n_s$  is given by the Lemma \ref{probability}, 3):
$$
\mathbb{E}^{\mu^u}[\rho/\mathcal{B}^n_s]=1_{\tau>t\wedge s}\alpha(t,u)^{-1}\alpha(t-t\wedge s,\xi(t\wedge s)), \ \mu^u-\mbox{a.s.}
$$
According to the Girsanov theorem \cite[Ch.8, Th.(1.4)]{Revuz_Yor}, the process
$$
\varphi(t,\omega)(\cdot)=w(\cdot)-\int^{\cdot\wedge t}_0 \alpha(t-s,\xi(s,w))^{-1} \nabla \alpha (t-s,\xi(s,w)) ds
$$
is the Wiener process with respect to $\pi^{u,t}.$ In other words, $\pi^{t,u}\circ \varphi(t,\cdot)^{-1}=\mu^u.$

\end{proof}

The Lemma \ref{flow} makes it possible to define stochastic integrals $\mathcal{J}^{u,\aleph,G}_{k}a$ for all $k\in \mathcal{K}_n $ and $a\in L^2(\mathcal{S}^d_+,\alpha(t_d,u)dt),$ $d=|k|,$  as follows 
\begin{equation}
\label{281}
\begin{gathered}
\mathcal{J}^{u,\aleph,G}_k a = \mathcal{I}^u_{k_d} (1_{\cdot<\tau}(h\circ \Phi)(\cdot))\bigg(= \int^{\tau}_0 (h\circ \Phi)(t) dw_{k_d}(t)\bigg),
\end{gathered}
\end{equation}
where
$$
h(t,\omega)=I^u_{k_1,\ldots ,k_{d-1}}(\tilde{a}(t))(\omega), 
$$
$$  
\tilde{a}(t)(t_1,\ldots,t_{d-1})=
\begin{cases} 
a(t_1,\ldots,t_{d-1},t), \ t_{d-1}<t, \\
0, t_{d-1}\geq t.
\end{cases}
$$
Composition $(h\circ \Phi)(t,\omega)$ equals $I^u_{k_1,\ldots ,k_{d-1}}(\tilde{a}(t))(\varphi(t,\omega)).$ Comparing this expression to \eqref{13}, one sees that the differentials in the inner $d-1$ integral are transformed according to $\varphi.$ Progressive measurability of the integrand in \eqref{281} follows from the Lemma \ref{flow} and the existence of a progressively measurable modification of $\{I^u_{k_1,\ldots ,k_{d-1}} (\tilde{a}(t)), t\geq 0\}$ \cite{Ondrejat}. Following calculation shows that $\mathcal{J}^{u,\aleph,G}_{k}$ is well-defined. 
\begin{equation}
\label{27}
\mathbb{E}^{\mu^u} \int^\infty_0 1_{\tau>t}( h\circ \Phi)(t)^2 dt=\int^\infty_0 \alpha(t,u) \mathbb{E}^{\pi^{t,u}_{\aleph,G}} (h\circ \Phi)(t)^2 dt=
\end{equation}

$$
=\int^\infty_0 \alpha(t,u) \mathbb{E}^{\mu^u}h(t)^2 dt
=\int_{\mathcal{S}^d_+} \alpha(t_d,u) a(t)^2dt.
$$

In the next Theorem we prove that operators $\mathcal{J}^{u,\aleph,G}_k$ constitute the It\^o-Wiener expansion of the space $L^2(\mathcal{C}^n,\mathcal{B}^n_{\tau},\mu^u).$

\begin{thm}
\label{stopped}
1) Spaces $\mathcal{J}^{u,\aleph,G}_k(L^2(\mathcal{S}^{|k|}_+,\alpha(t_{|k|},u)dt))$ corresponding to different $k\in\mathcal{K}_n$ are pairwise orthogonal;

2) $\mathcal{J}^{u,\aleph,G}_k$ is an isometry of $L^2(\mathcal{S}^{|k|}_+,\alpha(t_{|k|},u)dt)$ into $L^2(\mathcal{C}^n,\mathcal{B}^n_{\tau},\mu^u);$

3) $L^2(\mathcal{C}^n,\mathcal{B}^n_{\tau},\mu^u)=\oplus_{k\in \mathcal{K}_n} \mathcal{J}^{u,\aleph,G}_k (L^2(\mathcal{S}^{|k|}_+,\alpha(t_{|k|},u)dt)).$

\end{thm}

\begin{proof}
Properties 1) and 2) immediately follow from the calculation \eqref{27}.

 Consider $f\in L^2(\mathcal{C}^n,\mathcal{B}^n_{\tau},\mu^u)$ with $\mathbb{E}^{\mu^u}f=0.$ According to the Clark representation theorem \cite[Ch.5, Th.(3.5)]{Revuz_Yor} and \eqref{02-2}
$$
f=\sum^n_{j=1} \int^{\tau}_0 \mathcal{Q}^u_j f (t)dw_j(t), \ \mu^u-\mbox{a.s.},
$$
and
$$
\mathbb{E}^{\mu^u}f^2=\sum^n_{j=1} \int^\infty_0 \mathbb{E}^{\mu^u} 1_{\tau>t} \mathcal{Q}^u_j f(t)^2dt.
$$
Consider the progressively measurable process $((\mathcal{Q}^u_j f )\circ \Psi)(t).$ Next identities follow from the Lemma \ref{flow}.
$$
\int^\infty_{0} \alpha(t,u) \mathbb{E}^{\mu^u} ((\mathcal{Q}^u_j f)\circ \Psi)(t)^2dt=\int^\infty_{0} \alpha(t,u) \mathbb{E}^{\pi^{t,u}} \mathcal{Q}^u_j f(t)^2dt=
$$
$$
=\int^\infty_0 \mathbb{E}^{\mu^u} 1_{\tau>t}\mathcal{Q}^u_j f(t)^2dt<\infty.
$$
Hence, $(\mathcal{Q}^u_j f )\circ \Psi$ can be viewed as a measurable mapping of $\mathbb{R}_+$ into $L^2(\mathcal{C}^n,\mathcal{B}^n,\mu^u).$ The It\^o-Wiener expansion then produces kernels $b_{k,j}(t)\in L^2(\mathcal{S}^{|k|}_+),$ $k\in \mathcal{K}_n$ such that for a.a. $t>0$
\begin{equation}
\label{210}
((\mathcal{Q}^u_j f )\circ \Psi)(t)=\sum_{k\in \mathcal{K}_n} I^u_{k} (b_{k,j}(t)) \ \mbox{in} \ L^2(\mathcal{C}^n,\mathcal{B}^n,\mu^u).
\end{equation}
In fact \eqref{02-5},

\begin{equation}
\label{04-2}
b_{k,j}(t)=Q^u_k (((\mathcal{Q}^u_j f )\circ \Psi)(t)).
\end{equation}

From the Lemma \ref{measurability1} it follows that $b_{k,j}$ can be chosen as measurable functions $(t_1,\ldots,t_{|k|},t)\to b_{k,j}(t)(t_1,\ldots,t_{|k|}).$ Put
$$
a_{k,j}(t_1,\ldots,t_{|k|},t)=b_{k,j}(t)(t_1,\ldots,t_{|k|}).
$$
Then
\begin{equation}
\label{28}
\mathbb{E}^{\mu^u}f^2 =\sum^n_{j=1} \sum_{k\in \mathcal{K}_n} \int^\infty_{0} \alpha(t,u) \int_{\mathcal{S}^{|k|}_+}  b_{k,j}(t)(s)^2dsdt=
\end{equation}
$$
=\sum_{k\in \mathcal{K}_n,|k|\geq 1} \int_{\mathcal{S}^{|k|}_+} \alpha(t_{|k|},u) a_{k}(t)^2dt.
$$
In particular, $a_k\in L^2(\mathcal{S}^{|k|}_+,\alpha(t_{|k|},u)dt)$ and integrals $\mathcal{J}^{u,\aleph,G}_ka_k$ are well-defined. Also \eqref{28} shows that the series $\sum_{k\in\mathcal{K}_n} \mathcal{J}^{u,\aleph,G}_{k,j} a_{k,j}$ converges in $L^2(\mathcal{C}^n,\mathcal{B}^n_{\tau},\mu^u).$ It remains to show that
\begin{equation}
\label{29}
\sum_{k\in\mathcal{K}_n} \mathcal{J}^{u,\aleph,G}_{k,j} a_{k,j}= \int^{\tau}_0 \mathcal{Q}^u_j f(t)dw_j(t).
\end{equation}
Denote $h_{k,j}(\omega,t)=I^u_k(b_{k,j}(t))(\omega).$ From \eqref{210} it follows that  $\mathcal{Q}^u_j f \circ \Psi=\sum_{k\in \mathcal{K}_n} h_{k,j}$ in $L^2(\mathbb{R}_+\times\mathcal{C}^n,\mathcal{P},\alpha(t,u)dt\times \mu^u(d\omega)).$ Finally, for each
$g\in L^2(\mathbb{R}_+\times \mathcal{C}^n,\mathcal{P},1_{\tau(\omega)>t}dt\times \mu^u(d\omega))$ one has
$$
\mathbb{E}^{\mu^u} \int^{\tau}_0 \mathcal{Q}^u_j f(t)dw_j(t) \int^{\tau}_0 g(t)dw_j(t) =\int^\infty_0 \mathbb{E}^{\mu^u}1_{\tau>t}\mathcal{Q}^u_j f(t)g(t)dt=
$$
$$
=\int^\infty_0 \mathbb{E}^{\mu^u}\alpha(t,u) (\mathcal{Q}^u_j f\circ \Psi)(t)(g\circ\Psi)(t)dt=
$$
$$
=\sum_{k\in\mathcal{K}_n} \int^\infty_0 \mathbb{E}^{\mu^u}\alpha(u,t) h_{k,j}(t)(g\circ\Psi)(t)dt=
$$
$$
=\sum_{k\in\mathcal{K}_n} \int^\infty_0 \mathbb{E}^{\mu^u} 1_{\tau>t} (h_{k,j}\circ\Phi)(t)g(t)dt=\sum_{k\in\mathcal{K}_n} \mathbb{E}^{\mu^u} \mathcal{J}^{u,\aleph,G}_{k,j}a_{k,j}\int^{\tau}_0 g(t)dw_j(t)
$$
and \eqref{29} is verified.

\end{proof}
Consider the case $\aleph=0.$ Then the lifetime $\tau$ of the solution to \eqref{30_1}  is the moment when the Brownian motion $w$ has left $G,$ and the process $\xi$ in \eqref{30_1} coincides with $w$ up to the moment $\tau.$ To separate this particular case, we will denote $\tau$ by $\tau_G$ and abbreviate $\mathcal{J}^{u,0,G}_k$ to $\mathcal{J}^{u,G}_k.$  Next we generalize Theorem \ref{stopped} to some measures on $(\mathcal{C}^n,\mathcal{B}^n_{\tau_G})$ that are absolutely continuous with respect to $\mu^u.$

Assume that $\varkappa\ll \mu^u$ on $\mathcal{B}^n_{\tau_G}$ with
$\rho=\frac{d\varkappa}{d\mu^u}.$ Let the Clark representation \cite[Ch.5, Th.(3.5)]{Revuz_Yor} of $\rho$ be
$$
\rho=1+\int^{\tau_G}_0 h(s)dw(s), \ \mu^u-\mbox{a.s.},
$$
for some progressively measurable $\mathbb{R}^n-$valued process $h.$ Denote
$$
\rho(t)=1+\int^{t\wedge\tau_G}_0h(s)dw(s).
$$
According to the Girsanov theorem \cite[Ch.8, Th.(1.4)]{Revuz_Yor} the process
\begin{equation}
\label{31_1}
\mathcal{G}(t,\omega)=\omega(t\wedge \tau_G(\omega))-\int^{t\wedge \tau_G(\omega)}_0 \frac{h(s,\omega)}{\rho(s,\omega)}ds
\end{equation}
on the probability space $(\mathcal{C}^n,\mathcal{B}^n_{\tau_G},\varkappa)$ is a continuous square integrable martingale with $\mathcal{G}(0)=u$ and
\begin{equation}
\label{31_2}
<\mathcal{G}_i,\mathcal{G}_j>(t)=\delta_{i,j}t\wedge \tau_G.
\end{equation}

\begin{lem}
\label{stopped_absolutely_continuous}
Assume that in \eqref{31_1} $\frac{h(s,\omega)}{\rho(s,\omega)}=\aleph(\omega(s))$ for some infinitely differentiable vector field $\aleph:G\to \mathbb{R}^n.$ Denote $\sigma$ the lifetime of the solution to \eqref{30_1} corresponding to $\aleph$. Then $\mathcal{G}$ is the measurable isomorphism of the space $(\mathcal{C}^n,\mathcal{B}^n_{\tau_G},\varkappa)$ onto the space $(\mathcal{C}^n,\mathcal{B}^n_{\sigma},\mu^u).$ In particular, the correspondence $f\to f\circ \mathcal{G}$ is an isometry of the space $L^2(\mathcal{C}^n,\mathcal{B}^n_{\sigma},\mu^u)$ onto the space $L^2(\mathcal{C}^n,\mathcal{B}^n_{\tau_G},\varkappa),$ and operators
$$
a_k\to (\mathcal{J}^{u,\aleph,G}_ka_k)\circ \mathcal{G}, \ k\in\mathcal{K}_n, a_k\in L^2(\mathcal{S}^{|k|}_+,\alpha_{\aleph,G}(t_{|k|},u)dt)
$$
possess following properties

\noindent
1) spaces $\mathcal{J}^{u,\aleph,G}_k(L^2(\mathcal{S}^{|k|}_+,\alpha_{\aleph,G}(t_{|k|},u)dt))\circ \mathcal{G}$ corresponding to different $k\in\mathcal{K}_n$ are pairwise orthogonal;

\noindent
2) $a_k\to (\mathcal{J}^{u,\aleph,G}_ka_k)\circ\mathcal{G}$ is an isometry of the space $L^2(\mathcal{S}^{|k|}_+,\alpha_{\aleph,G}(t_{|k|},u)dt)$ into the space $L^2(\mathcal{C}^n,\mathcal{B}^n_{\tau_G},\varkappa);$

\noindent
3) $L^2(\mathcal{C}^n,\mathcal{B}^n_{\tau_G},\varkappa)=\oplus_{k\in \mathcal{K}_n}( \mathcal{J}^{u,\aleph,G}_k (L^2(\mathcal{S}^{|k|}_+,\alpha_{\aleph,G}(t_{|k|},u)dt))\circ \mathcal{G}).$

\end{lem}

\begin{proof}
Denote $\xi$ the solution to \eqref{30_1} defined up to $\sigma.$ The assumption of the Lemma imply that
$$
\omega(t)=\mathcal{G}(t,\omega)+\int^t_0 \aleph(\omega(s))ds, \ t<\tau_G(\omega).
$$
Hence, $\omega(t)=\xi(t,\mathcal{G}(\cdot,\omega)),$ $t<\tau_G(\omega)$ and $\tau_G(\omega)=\sigma(\mathcal{G}(\cdot,\omega)).$ From \eqref{31_2} it follows that the distribution of $\mathcal{G}$ under the measure $\varkappa$ coincides with the distribution of the Brownian motion stopped at the moment $\sigma.$ That is,
$$
\varkappa\circ \mathcal{G}^{-1}=\mu^u \ \mbox{on} \ \mathcal{B}^n_{\sigma}.
$$
$\mathcal{G}$ is an isomorphism, as it is $\varkappa-$a.s. invertible. Actually, its inverse is  $\xi.$

\end{proof}

\section{$n-$point Motion of the Arratia Flow}

The main result of this section is an analogue of the It\^o-Wiener expansion for the space $L^2(\mathcal{C}^n,\mathcal{B}^n,\nu^u),$ where $u\in \mathcal{S}^n,$ $n\geq 1,$ $\nu^u$ is the distribution of the $n-$point motion $x_u$ of the Arratia flow. Denote $\tau_{\mathcal{S}^{n}}$ the first moment when two of the components of $x_u$ have met each other (and coalesced) and $p^{n-1}\in\mathcal{S}^{n-1}$ is the vector of positions of the remaining particles ($p^{n-1}_1<\ldots <p^{n-1}_{n-1}$  and $\{x(u_1,\tau_{\mathcal{S}^{n}}),\ldots,x(u_{n},\tau_{\mathcal{S}^{n}})\}=\{p^{n-1}_1,\ldots ,p^{n-1}_{n-1}\}$). The space $(\mathcal{C}^{n},\mathcal{B}^{n},\nu^u)$ is identified with the space with mixture of measures $(\mathcal{C}^{n-1}\times\mathcal{C}^{n}, \mathcal{B}^{n-1}\times \mathcal{B}^{n}_{\tau_{\mathcal{S}^{n}}},\nu^{p^{n-1}(\omega^{n})}(d\omega^{n-1})\mu^u(d\omega^{n}))$ (see the Introduction). Accordingly, the function $f\in L^2(\mathcal{C}^{n},\mathcal{B}^{n},\nu^u)$ will be considered as a function of $(\omega^{n-1},\omega^n),$ where $\omega^{n-1}$ refers to the $n-1$ trajectory left after the first coalescence and $\omega^n$ refers to the trajectories before the first coalescence. Using this representation, the It\^o-Wiener expansion will be constructed inductively. 

When $n=1,$ $\nu^u=\mu^u$ and the It\^o-Wiener expansion of $L^2(\mathcal{C}^1,\mathcal{B}^1,\nu^u)$ is determined by the operators $(I^u_k)_{k\in\mathcal{K}_1}$ of multiple stochastic integration with respect to the Brownian motion  \eqref{13}. So we set 
\begin{equation}
\label{25_02_5}
\mathcal{A}^u_k=I^u_k, \ k\in \mathcal{K}_1.
\end{equation} 
Then
$$
\mathcal{A}^u_k: L^2(\mathcal{S}^{|k|}_+,dt)\to L^2(\mathcal{C}^1,\mathcal{B}^1,\nu^u)
$$
is the isometry, $\mathcal{A}^u_k a\perp \mathcal{A}^u_l b$ for $k\ne l,$ and
$$
L^2(\mathcal{C}^1,\mathcal{B}^1,\nu^u)=\oplus_{k\in \mathcal{K}_1} \mathcal{A}^u_k( L^2(\mathcal{S}^{|k|}_+,dt)).
$$

To illuminate the key moments of the inductive construction let us consider some partial cases. We won't justify statements concerning measurability as they will be proved in the general case. When $n=2,$ $u\in \mathcal{S}^2,$ the only trajectory remained after the coalescence is the Brownian motion that is independent on the trajectories before the coalescence. Hence, it is natural to construct the It\^o-Wiener expansion combining the usual It\^o-Wiener expansion for the Brownian motion (via operators $I^v_k$ in \eqref{02-3}) with the one for the two-dimensional Brownian motion stopped when its trajectories meet each other (via operators $\mathcal{J}^{u,\mathcal{S}^2}_k$ in \eqref{281}). Define operators $\mathcal{A}^u_{k^1,k^2}$ as follows:
\begin{equation}
\label{30_03_2}
\mathcal{A}^u_{k^1,k^2} a(\omega^1,\omega^2)= I^{p^1(\omega^2)}_{k^1}(\mathcal{J}^{u,\mathcal{S}^2}_{k^2}a(t^1,\cdot)(\omega^2))(\omega^1).
\end{equation}
Then the squared norm of $\mathcal{A}^u_{k^1,k^2} a$  equals
\begin{equation}
\label{30_03_1}
\begin{gathered}
\int_{\mathcal{C}^2}\int_{\mathcal{C}^1} \mathcal{A}^u_{k^1,k^2} a(\omega^1,\omega^2)^2\mu^{p^1(\omega^2)}(d\omega^1)\mu^u(d\omega^2)= \\
=\int_{\mathcal{C}^2}\int_{\mathcal{C}^1}I^{p^1(\omega^2)}_{k^1}(\mathcal{J}^{u,\mathcal{S}^2}_{k^2}a(t^1,\cdot)(\omega^2))(\omega^1)^2\mu^{p^1(\omega^2)}(d\omega^1)\mu^u(d\omega^2)= \\
=\int_{\mathcal{C}^2}\int_{\mathcal{S}^{|k^1|}_+}\mathcal{J}^{u,\mathcal{S}^2}_{k^2}a(t^1,\cdot)(\omega^2)^2 dt^1 \mu^u(d\omega^2)=
\\
=\int_{\mathcal{S}^{|k^1|}_+} \int_{\mathcal{C}^2} \mathcal{J}^{u,\mathcal{S}^2}_{k^2}a(t^1,\cdot)(\omega^2)^2 \mu^u(d\omega^2)dt^1=\\
=\int_{\mathcal{S}^{|k^1|}_+} \int_{\mathcal{S}^{|k^2|}_+} a(t^1,t^2)^2 \alpha_{\mathcal{S}^2}(t^2_{|k^2|},u) dt^2dt^1.
\end{gathered}
\end{equation}
It follows that $\mathcal{A}^u_{k^1,k^2}$ is an isometry of $L^2(\mathcal{S}^{|k^1|,|k^2|}_+,\alpha_{\mathcal{S}^2}(t^2_{|k^2|},u) dt^2dt^1)$ into the space $L^2(\mathcal{C}^2,\mathcal{B}^2,\nu^u).$ Also, from \eqref{30_03_1} and properties of operators $I^v_k$ and $\mathcal{J}^{u,\mathcal{S}^2}_k$ it follows that ranges  $\mathcal{A}^u_{k^1,k^2}(L^2(\mathcal{S}^{|k^1|,|k^2|}_+,\alpha_{\mathcal{S}^2}(t^2_{|k^2|},u) dt^2dt^1))$ corresponding to different indices $(k^1,k^2)$ are orthogonal. Hence, operators $\mathcal{A}^u_{k^1,k^2}$ may be viewed as analogues of operators of multiple stochastic integration. They indeed produce the It\^o-Wiener expansion for the two-point motion of the Arratia flow in the sense that 
$$
L^2(\mathcal{C}^2,\mathcal{B}^2,\nu^u)=\oplus_{(k^1,k^2)\in \mathcal{K}_{1,2}}
\mathcal{A}^u_{k^1,k^2}(L^2(\mathcal{S}^{|k^1|,|k^2|}_+,\alpha_{\mathcal{S}^2}(t^2_{|k^2|},u) dt^2dt^1))
$$
(see Theorem \ref{finite-dimensional} for the proof). 

Consider $u\in \mathcal{S}^3.$ To construct the It\^o-Wiener expansion for the 3-point motion $x_u$ we will use the same isomorphism as above
$$
(\mathcal{C}^{3},\mathcal{B}^{3},\nu^u)\simeq (\mathcal{C}^{2}\times\mathcal{C}^{3}, \mathcal{B}^{2}\times \mathcal{B}^{3}_{\tau_{\mathcal{S}^{3}}},\nu^{p^{2}(\omega^{3})}(d\omega^{2})\mu^u(d\omega^{3})).
$$
Using operators $\mathcal{A}^u_{k^1,k^2}$ \eqref{30_03_2}, each functional $f\in L^2(\mathcal{C}^3,\mathcal{B}^3,\nu^u)$ can be written in the form 
\begin{equation}
\label{31_03_1}
f(\cdot,\omega^3)=\sum_{(k^1,k^2)\in\mathcal{K}_{1,2}}\mathcal{A}^{p^2(\omega^3)}_{k^1,k^2}(a_{k^1,k^2}(\cdot,\omega^3)).
\end{equation}
According to \eqref{30_03_1} the squared norm of the summand $\mathcal{A}^{p^2}_{k^1,k^2}a_{k^1,k^2}$ equals
\begin{equation}
\label{31_03_3}
\int_{\mathcal{C}^3}\int_{\mathcal{S}^{|k^1|,|k^2|}_+}a_{k^1,k^2}(t^1,t^2,\omega^3)^2\alpha_{\mathcal{S}^{2}}(t^2_{|k^2|},p^2(\omega^3))dt^1dt^2\mu^u(d\omega^3),
\end{equation}
what means that $a_{k^1,k^2}(t^1,t^2,\cdot)\in L^2(\mathcal{C}^3,\mathcal{B}^3_{\tau_{\mathcal{S}^3}},\varkappa^{t^2_{|k^2|};u}),$ where for $t_2=t^2_{|k^2|}$ the measure $\varkappa^{t_2;u}$ is defined via the density
$$
\frac{d\varkappa^{t_2;u}}{d\mu^u}(\omega^3)=\frac{\alpha_{\mathcal{S}^3}(t_2,p^{2}(\omega^3))}{\beta_{t_2}(u)}, \ \beta_{t_2}(u)=\mathbb{E}^{\mu^u}\alpha_{\mathcal{S}^3}(t_2,p^{2}).
$$
Hence, to expand further $a_{k^1,k^2}(t^1,t^2,\cdot)$ as a series of integrals, Lemma \ref{stopped_absolutely_continuous} have to be used. Put $\rho_{t_2}(u)=\alpha_{\mathcal{S}^3}(t_2,u)$. It will be proved in the Lemma \ref{add1} that the Clark representation formula for $\rho_{t_2}(p^2)$ is
$$
\rho_{t_2}(p^{2})=\beta_{t_2}(u)+\int^{\tau_{\mathcal{S}^{3}}}_0 \nabla \beta_{t_2}(w(r))dw(r), \ \mu^u-\mbox{a.s.}
$$
Also, $\mathbb{E}^{\mu^u}[\rho_{t_2}(p^2)/\mathcal{B}^3_r]=\beta_{t_2}(w(r\wedge\tau_{\mathcal{S}^{3}})).$ Consequently, the vector field $\aleph$ in the Lemma \ref{stopped_absolutely_continuous} equals 
$$
\aleph_{t_2}(u)=\nabla \log \beta_{t_2}(u).
$$
From the Lemma \ref{stopped_absolutely_continuous} it follows that operators 
$$
a_k\to (\mathcal{J}^{u,\aleph_{t_2,G}}_ka_k)\circ \mathcal{G}^{t_2}, \ k\in\mathcal{K}_3, 
$$
constitute the It\^o-Wiener expansion of the space $L^2(\mathcal{C}^3,\mathcal{B}^3_{\tau_{\mathcal{S}^3}},\varkappa^{t_2;u}),$ where 
$$
\mathcal{G}^{t_2}(\omega^3)=\omega^3(\cdot\wedge \tau_{\mathcal{S}^{3}}(\omega^3))-\int^{\cdot\wedge \tau_{\mathcal{S}^{3}}(\omega^3)}_0 \aleph_{t_2}(\omega^3(r))dr.
$$
Accordingly, functions $a_{k^1,k^2}(t^1,t^2,\cdot)$ can be expanded further: 
\begin{equation}
\label{31_03_2}
a_{k^1,k^2}(t^1,t^2,\cdot)=\sum_{k^3 \in \mathcal{K}_3} (\mathcal{J}^{u,\aleph_{t^2_{|k^2|}},G}_{k^3}a_{k^1,k^2,k^3}(t^1,t^2,\cdot))\circ \mathcal{G}^{t^2_{|k^2|}}.
\end{equation}
Equations \eqref{31_03_1},\eqref{31_03_2} suggest that for the case $n=3,$ $u\in \mathcal{S}^3$ operators $\mathcal{A}^u_{k^1,k^2,k^3}$ have to be defined as
$$
\mathcal{A}^u_{k^1,k^2,k^3} a (\omega^2,\omega^3)= \mathcal{A}^{p^2(\omega^3)}_{k^1,k^2} \bigg(\mathcal{J}^{u,\aleph_{t^2_{|k^2|}},G}_{k^3}a(t^1,t^2,\cdot)(\mathcal{G}^{t^2_{|k^2|}}(\omega^3))\bigg)(\omega^2).
$$
In our main result (Theorem \ref{finite-dimensional}) we prove that these operators indeed lead to the It\^o-Wiener expansion for the 3-point motion $x_u.$  Let us calculate the squared norm of  $\mathcal{A}^u_{k^1,k^2,k^3} a.$
$$
\int_{\mathcal{C}^3}\int_{\mathcal{C}^2} \mathcal{A}^u_{k^1,k^2,k^3} a(\omega^2,\omega^3)^2\nu^{p^2(\omega^3)}(d\omega^2)\mu^u(d\omega^3)=
$$
$$
=\int_{\mathcal{C}^3}\int_{\mathcal{C}^2}\mathcal{A}^{p^2(\omega^3)}_{k^1,k^2} \bigg(\mathcal{J}^{u,\aleph_{t^2_{|k^2|}},G}_{k^3}a(t^1,t^2,\cdot)(\mathcal{G}^{t^2_{|k^2|}}(\omega^3))\bigg)(\omega^2)^2 d\nu^{p^2(\omega^3)}d\mu^u=
$$
$$
=\int_{\mathcal{C}^3}\int_{\mathcal{S}^{|k^1|,|k^2|}_+} 
\mathcal{J}^{u,\aleph_{t^2_{|k^2|}},G}_{k^3}a(t^1,t^2,\cdot)(\mathcal{G}^{t^2_{|k^2|}}(\omega^3))^2 \alpha_{\mathcal{S}^{2}}(t^2_{|k^2|},p^2(\omega^3))dt^1 dt^2 d\mu^u=
$$
$$
=\int_{\mathcal{S}^{|k^1|,|k^2|}_+} \beta_{t^2_{|k^2|}}(u) 
\int_{\mathcal{C}^3} \mathcal{J}^{u,\aleph_{t^2_{|k^2|}},G}_{k^3}a(t^1,t^2,\cdot)(\mathcal{G}^{t^2_{|k^2|}}(\omega^3))^2  d\varkappa^{t^2_{|k^2|};u} dt^1 dt^2 =
$$
$$
=\int_{\mathcal{S}^{|k^1|,|k^2|,|k^3|}_+} a(t^1,t^2,t^3)^2   \beta_{t^2_{|k^2|}}(u)\alpha_{\aleph_{t^2_{|k^2|}},\mathcal{S}^3}(t^3_{|k^3|},u)   dt^1 dt^2 dt^3.
$$
In the last equality Lemma \ref{stopped_absolutely_continuous} was used.  Comparing this formula to \eqref{31_03_3} note that the density $\alpha_{\mathcal{S}^2}(t^2_{|k^2|},u)$  have changed to $ \beta_{t^2_{|k^2|}}(u)\alpha_{\aleph_{t^2_{|k ^2|}},\mathcal{S}^3}(t^3_{|k^3|},u).$ Relying on this observation we introduce functions  $\rho_{t_2,\ldots,t_n}(u),$ that will appear as densities in the general case. Simultaneously, functions 
$\beta_{t_2,\ldots,t_{n-1}}(u),$ $\aleph_{t_2,\ldots,t_{n-1}}(u)$ are defined.
\begin{equation}
\label{rho_eq}
\begin{gathered}
\rho_{t_2}(u)=\alpha_{\mathcal{S}^2}(u,t_2), \ u\in \mathcal{S}^2;\\
\beta_{t_2,\ldots,t_{n-1}}(u)=\mathbb{E}^{\mu^u} \rho_{t_2,\ldots,t_{n-1}}(p^{n-1}), \ \aleph_{t_2,\ldots,t_{n-1}}(u)=\nabla \log \beta_{t_2,\ldots,t_{n-1}}(u),\\
\rho_{t_2,\ldots,t_n}(u)=\alpha_{\aleph_{t_2,\ldots,t_{n-1}},\mathcal{S}^n}(t_n,u) \beta_{t_2,\ldots,t_{n-1}}(u), \ u\in \mathcal{S}^n,
\end{gathered}
\end{equation}
where $ t_2,\ldots ,t_{n-1}>0$ and $\alpha_{\aleph,G}$ is defined in \eqref{25_02_3}. Given $u\in \mathcal{S}^n$ and positive $t_2,\ldots ,t_{n-1},$ consider the measure $\varkappa^{t_2,\ldots,t_{n-1};u}$ on $(\mathcal{C}^{n},\mathcal{B}^{n}_{\tau_{\mathcal{S}^{n}}})$ defined via the density
$$
\frac{d\varkappa^{t_2,\ldots,t_{n-1};u}}{d\mu^u}(\omega^n)=\frac{\rho_{t_2,\ldots,t_{n-1}}(p^{n-1}(\omega^n))}
{\beta_{t_2,\ldots,t_{n-1}}(u)}.
$$
As it is seen from the case $n=3,$ the orthogonal structure of the space $L^2(\mathcal{C}^n,\mathcal{B}^{n}_{\tau_{\mathcal{S}^{n}}}, \varkappa^{t_2,\ldots,t_{n-1};u})$ is needed.

\begin{lem}
\label{add1}
1) The Clark representation formula for $\rho_{t_2,\ldots,t_{n-1}}(p^{n-1})$ is
$$
\rho_{t_2,\ldots,t_{n-1}}(p^{n-1})=\beta_{t_2,\ldots,t_{n-1}}(u)+\int^{\tau_{\mathcal{S}^{n}}}_0 \nabla \beta_{t_2,\ldots,t_{n-1}}(w(r))dw(r), \ \mu^u-\mbox{a.s.}
$$

2) $\mathbb{E}^{\mu^u}[\rho_{t_2,\ldots,t_{n-1}}(p^{n-1})/\mathcal{B}^n_r]=\beta_{t_2,\ldots,t_{n-1}}(w(r\wedge\tau_{\mathcal{S}^{n}})), \ \mu^u-\mbox{a.s.}$

3) The mapping
$$
\mathcal{G}^{t_2,\ldots,t_{n-1}}(\omega^n)=\omega^n(\cdot\wedge \tau_{\mathcal{S}^{n}}(\omega^n))-\int^{\cdot\wedge \tau_{\mathcal{S}^{n}}(\omega^n)}_0 \aleph_{t_2,\ldots,t_{n-1}}(\omega^n(r))dr
$$
is the measurable isomorphism of the space $(\mathcal{C}^n,\mathcal{B}^n_{\tau_\mathcal{S}^n},\varkappa^{t_2,\ldots,t_{n-1};u})$ onto the space $(\mathcal{C}^n,\mathcal{B}^n_{\tau_{\aleph_{t_2,\ldots,t_{n-1}},\mathcal{S}^n}},\mu^u).$ 
In particular, operators
\begin{equation}
\label{31_03_11}
a\to (\mathcal{J}^{u,\aleph_{t_2,\ldots,t_{n-1}},\mathcal{S}^n}_k a)\circ \mathcal{G}^{t_2,\ldots,t_{n-1}}, \ k\in\mathcal{K}_n, a\in L^2(\mathcal{S}^{|k|}_+,\alpha_{\aleph_{t_2,\ldots,t_{n-1}},\mathcal{S}^n}(t_{|k|},u)dt)
\end{equation}
possess following properties

\noindent
3.1) each operator $a\to(\mathcal{J}^{u,\aleph_{t_2,\ldots,t_{n-1}},\mathcal{S}^n}_ka)\circ\mathcal{G}^{t_2,\ldots,t_{n-1}}$ is an isometry of the space
$L^2(\mathcal{S}^{|k|}_+,\alpha_{\aleph_{t_2,\ldots,t_{n-1}},\mathcal{S}^n}(t_{|k|},u)dt)$ into the space $L^2(\mathcal{C}^n,\mathcal{B}^n_{\tau_{\mathcal{S}^n}},\varkappa^{t_2,\ldots,t_{n-1};u});$

\noindent
3.2) spaces $\mathcal{J}^{u,\aleph_{t_2,\ldots,t_{n-1}},\mathcal{S}^n}_k(L^2(\mathcal{S}^{|k|}_+,\alpha_{\aleph_{t_2,\ldots,t_{n-1}},\mathcal{S}^n}(t_{|k|},u)dt))\circ \mathcal{G}^{t_2,\ldots,t_{n-1}}$ cor\-res\-pon\-ding to different $k\in\mathcal{K}_n$ are pairwise orthogonal;

\noindent
3.3) $L^2(\mathcal{C}^n,\mathcal{B}^n_{\tau_{\mathcal{S}^n}},\varkappa^{t_2,\ldots,t_{n-1};u})=$
$$
=\oplus_{k\in \mathcal{K}_n}( \mathcal{J}^{u,\aleph_{t_2,\ldots,t_{n-1}},\mathcal{S}^n}_k (L^2(\mathcal{S}^{|k|}_+,\alpha_{\aleph_{t_2,\ldots,t_{n-1}},\mathcal{S}^n}(t_{|k|},u)dt))\circ \mathcal{G}^{t_2,\ldots,t_{n-1}}).
$$

\end{lem}

\begin{proof}
During the proof we omit indices $t_2,\ldots,t_{n-1}$ and put $G=\mathcal{S}^n.$ The function $\beta(u)=\mathbb{E}^{\mu^u} \rho(p^{n-1})$ is harmonic in $G$ \cite[Th.11.1.17]{Stroock_Probability}. Denote $G^{\delta}$ the set of points of $G$ whose distance to $\partial G$ exceeds $\delta.$ From the It\^o's formula applied to the function $\beta(w(t\wedge \tau_{G^\delta}))$ it follows that $\mu^u-$a.s.
$$
\beta(w(t\wedge \tau_{G^\delta}))=\beta(u)+
\int^{t\wedge\tau_{G^\delta}}_0 \nabla \beta(w(r))dw(r).
$$
Letting $t\to \infty$ and $\delta\to 0,$ part 1) follows. 

Part 2) follows from the Markov property of the Wiener process:
$$
\mathbb{E}^{\mu^u}[\rho(p^{n-1})/\mathcal{B}^n_r]=1_{\tau_G <r} \rho(p^{n-1})+
1_{\tau_G>r} \mathbb{E}^{\mu^v}(\rho(p^{n-1}))|_{v=w(r)}=\beta(w(r\wedge\tau_G)).
$$

Finally, part 3) is an immediate application of the Lemma \ref{stopped_absolutely_continuous} and formulas from first two parts. 

\end{proof}

\begin{cor}
\label{add2}
There exist operators
\begin{equation}
\label{3112}
\mathcal{R}^{t_2,\ldots,t_{n-1};u}_k : L^2(\mathcal{C}^n,\mathcal{B}^n_{\tau_{\mathcal{S}^n}},\varkappa^{t_2,\ldots,t_{n-1};u})\to L^2(\mathcal{S}^{|k|}_+,\alpha_{\aleph_{t_2,\ldots,t_{n-1}},\mathcal{S}^n}(t_{|k|},u)dt),
\end{equation}
such that each $g\in L^2(\mathcal{C}^n,\mathcal{B}^n_{\tau_{\mathcal{S}^n}},\varkappa^{t_2,\ldots,t_{n-1};u})$ has a series representation
$$
g=\sum_{k\in \mathcal{K}_n} (\mathcal{J}^{u,\aleph_{t_2,\ldots,t_{n-1}},\mathcal{S}^n}_k(\mathcal{R}^{t_2,\ldots,t_{n-1};u}_k g)) \circ\mathcal{G}^{t_2,\ldots,t_{n-1}}.
$$
In fact, $\mathcal{R}^{t_2,\ldots,t_{n-1};u}_k$ is a composition of the orthogonal projection of the space $L^2(\mathcal{C}^n,\mathcal{B}^n_{\tau_{\mathcal{S}^n}}, \varkappa^{t_2,\ldots,t_{n-1};u})$ 
onto $\mathcal{J}^{u,\aleph_{t_2,\ldots,t_{n-1}},\mathcal{S}^n}_k(L^2(\mathcal{S}^{|k|}_+,\alpha_{\aleph_{t_2,\ldots,t_{n-1}},\mathcal{S}^n}(t_{|k|},u)dt))\circ \mathcal{G}^{t_2,\ldots,t_{n-1}}$ 
with the inverse  $(\mathcal{J}^{u,\aleph_{t_2,\ldots,t_{n-1}},\mathcal{S}^n}_k)^{-1}.$
\end{cor}

Now we are in a position to describe the inductive construction.

\noindent
{\bf Induction base.} {\it For every $u\in \mathbb{R}$ and $k\in \mathcal{K}_1$ set   
$$
\mathcal{A}^u_k=I^u_k,
$$
where $I^u_k$ are operators of multiple stochastic integration with respect to the Brownian motion $w$ on $(\mathcal{C}^1,\mathcal{B}^1)$ \eqref{13}.
}

\noindent
{\bf Induction hypothesis.} {\it For every $v\in \mathcal{S}^{n-1}$ and $(k^1,\ldots ,k^{n-1})\in \mathcal{K}_{1,\ldots,n-1}$ an operator
$$
\mathcal{A}^v_{k^1,\ldots,k^{n-1}} : L^2(\mathcal{S}^{|k^1|,\ldots,|k^{n-1}|}_+,\rho_{t^2_{|k^2|},\ldots,t^{n-1}_{|k^{n-1}|}}(v)dt^1\ldots dt^{n-1}) \to
$$
$$
\to L^2(\mathcal{C}^{n-1},\mathcal{B}^{n-1},\nu^v)
$$
is defined in such a way that

\noindent
{\bf (H1)} $\mathcal{A}^v_{k^1,\ldots,k^{n-1}}$ is the isometry;

\noindent
{\bf (H2)} spaces $\mathcal{A}^v_{k^1,\ldots,k^{n-1}}(L^2(\mathcal{S}^{|k^1|,\ldots,|k^{n-1}|}_+,\rho_{t^2_{|k^2|},\ldots,t^{n-1}_{|k^{n-1}|}}(v)dt^1\ldots dt^n)),$ that correspond to different indices $(k^1,\ldots,k^{n-1})\in \mathcal{K}_{1,\ldots,n-1},$ are pairwise orthogonal;

\noindent
{\bf (H3)} $L^2(\mathcal{C}^{n-1},\mathcal{B}^{n-1},\nu^v)=$
$$
=\oplus_{(k^1,\ldots,k^{n-1})\in \mathcal{K}_{1,\ldots,n-1}}\mathcal{A}^v_{k^1,\ldots,k^{n-1}}(L^2(\mathcal{S}^{|k^1|,\ldots,|k^{n-1}|}_+,\rho_{t^2_{|k^2|},\ldots,t^{n-1}_{|k^{n-1}|}}(v)dt^1\ldots dt^{n-1}));
$$
respectively, projections
$$
P^v_{k^1,\ldots,k^{n-1}}:L^2(\mathcal{C}^{n-1},\mathcal{B}^{n-1},\nu^v)\to L^2(\mathcal{S}^{|k^1|,\ldots,|k^{n-1}|}_+,\rho_{t^2_{|k^2|},\ldots,t^{n-1}_{|k^{n-1}|}}(v)dt^1\ldots dt^n)
$$ are defined in such a way that for each $g\in L^2(\mathcal{C}^{n-1},\mathcal{B}^{n-1},\nu^v)$ an equality holds
$$
g=\sum_{(k^1,\ldots,k^{n-1})\in \mathcal{K}_{1,\ldots,n-1}} \mathcal{A}^v_{k^1,\ldots,k^{n-1}}P^v_{k^1,\ldots,k^{n-1}}g;
$$

\noindent
{\bf (H4)} for any $(k^1,\ldots,k^{n-1})\in \mathcal{K}_{1,\ldots,n-1}$ and bounded Borel function with compact support $a:\mathcal{S}^{|k^1|,\ldots,|k^{n-1}|}_+\to \mathbb{R},$ a family $\{\mathcal{A}^v_{k^1,\ldots ,k^{n-1}} a\}_{v\in \mathcal{S}^{n-1}}$ can be realized as a measurable function on $\mathcal{C}^{n-1}\times \mathcal{S}^{n-1}$ w.r.t. the family $\{\nu^v\}_{v\in \mathcal{S}^{n-1}}$ (Definition \ref{measurable_family});

\noindent
{\bf (H5)} for any $(k^1,\ldots,k^{n-1})\in \mathcal{K}_{1,\ldots,n-1}$ and bounded Borel function $g: \break \mathcal{C}^{n-1}\to \mathbb{R},$ a family $\{P^v_{k^1,\ldots ,k^{n-1}} g\}_{v\in \mathcal{S}^{n-1}}$ can be realized as a me\-a\-su\-ra\-ble func\-tion on $\mathcal{S}^{|k^1|,\ldots,|k^{n-1}|}_+\times \mathcal{S}^{n-1}$ with respect to the family of measures
$$
\{\rho_{t^2_{|k^2|},\ldots,t^{n-1}_{|k^{n-1}|}}(v)dt^1\ldots dt^{n-1}\}_{v\in \mathcal{S}^{n-1}}
$$ 
(Definition \ref{measurable_family}). }

\noindent
{\bf Induction step.} {\it Consider $(k^1,\ldots,k^n)\in \mathcal{K}_{1,\ldots,n}.$ For every 
$$
a\in L^2(\mathcal{S}^{|k^1|,\ldots,|k^{n}|}_+, \break \rho_{t^2_{|k^2|},\ldots,t^{n}_{|k^{n}|}}(u)dt^2\ldots dt^{n})
$$ 
denote $\mathcal{T}a$ the result of application an operator \eqref{31_03_11} to the last $|k^{n}|$ coordinates, that is $\mathcal{T}a(t^1,\ldots ,t^{n-1},\omega^{n})=$
\begin{equation}
\label{I_def}
= \mathcal{J}^{u,\aleph_{t^2_{|k^2|},\ldots,t^{n-1}_{|k^{n-1}|}},\mathcal{S}^n}_k(a(t^1,\ldots,t^{n-1},\cdot))(\mathcal{G}_{t^2_{|k^2|},\ldots,t^{n-1}_{|k^{n-1}|}}(\omega^{n})).
\end{equation}
Next, define operators $\mathcal{A}^u_{k^1,\ldots,k^{n}}a$  for $u\in \mathcal{S}^n$ by the rule
\begin{equation}
\label{operators}
\mathcal{A}^u_{k^1,\ldots,k^{n}}a(\omega^{n-1},\omega^{n})=\mathcal{A}^{p^{n-1}(\omega^{n})}_{k^1,\ldots,k^{n-1}}(\mathcal{T}a(\cdot,\omega^{n}))(\omega^{n-1}).
\end{equation}
}

In the Theorem \ref{finite-dimensional} we will show that operators \eqref{operators} also satisfy conditions {\bf (H1)-(H5)}. It gives possibility to define operators $\mathcal{A}^u_{k^1,\ldots,k^n}$ with properties {\bf (H1)-(H5)} for all $n\geq 1,$ $u\in \mathcal{S}^n,$ $(k^1,\ldots,k^n)\in \mathcal{K}_{1,\ldots,n}.$ From the induction base and the induction step it is seen that $\mathcal{A}^u_{k_1,\ldots ,k_n}$ are operators of the multiple stochastic integration with respect to finite-point motion of the Arratia flow. Thus, properties {\bf (H1)-(H3)} state that these operators constitute an analogue of the It\^o-Wiener expansion.  Technical properties {\bf (H4)-(H5)} are called to justify measurability issues.  Indeed, due to the complicated expression in \eqref{operators}, its measurability in $(\omega^{n-1},\omega^n)$ is not obvious. To prove it we will need measurability properties of operators in \eqref{02-1},\eqref{02-2},\eqref{02-3},\eqref{02-5},   \eqref{I_def}, \eqref{31_03_11}, \eqref{3112}. 

The key instrument in deducing the existence of measurable realizations will be the Lemma \ref{measurability}. It states that under rather general assumptions on spaces $(\mathcal{X},\mathcal{B}_{\mathcal{X}},\mu^\omega),$ $(\mathcal{Y},\mathcal{B}_{\mathcal{Y}},\nu^\omega),$ and operators 
$$
\mathcal{A}^\omega:L^2(\mathcal{X},\mathcal{B}_{\mathcal{X}},\mu^\omega)\to L^2(\mathcal{Y},\mathcal{B}_{\mathcal{Y}},\nu^\omega),
$$
the existence of measurable realizations for ``test'' families of the kind $\{\mathcal{A}^\omega f_0\}$ automatically implies the existence of measurable realizations for all families $\{\mathcal{A}^\omega f(\cdot,\omega)\},$ such that for each $\omega$ $\mathcal{A}^\omega f(\cdot,\omega)$ is well-defined. Accordingly, properties {\bf (H4)-(H5)} are immediately strengthened. For example, given measurable $a:\mathcal{S}^{|k^1|,\ldots,|k^{n-1}|}_+\times \Omega\to \mathbb{R}, $ $\xi:\Omega\to \mathcal{S}^{n-1}$ such that
$$
\forall \omega\in \Omega \ \ a(\cdot, \omega)\in L^2(\mathcal{S}^{|k^1|,\ldots,|k^{n-1}|}_+,\rho_{t^2_{|k^2|},\ldots,t^{n-1}_{|k^{n-1}|}}(\xi(\omega))dt^1\ldots dt^{n-1}),
$$
a family $\{\mathcal{A}^{\xi(\omega)}_{k^1,\ldots ,k^{n-1}} (a(\cdot,\omega))\}_{\omega\in \Omega}$ can be realized as a measurable function on $\mathcal{C}^{n-1}\times \Omega$ w.r.t. the family $\{\nu^{\xi(\omega)}\}_{\omega\in \Omega}.$ For the proof note that

1) there exists a sequence of bounded Borel functions with compact support $(f_n)$ on $\mathcal{S}^{|k^1|,\ldots,|k^{n-1}|}_+,$ which is total in $L^2$ relatively to any Radon measure;

2) there exists a sequence  of bounded Borel functions $(g_n)$ on $\mathcal{C}^{n-1},$ which is total in $L^2$ relatively to any probability measure;

3) in the view of {\bf (H4)} each family $\{\mathcal{A}^{\xi(\omega)}_{k^1,\ldots ,k^{n-1}}f_n\}_{\omega\in \Omega}$ can be realized as a measurable function on $\mathcal{C}^{n-1}\times \Omega$  w.r.t. the family $\{\nu^{\xi(\omega)}\}_{\omega\in \Omega}.$

Consequently, the Lemma \ref{measurability} gives the needed result.

In the case $n=1,$ operators $\mathcal{A}^u_k$ coincide with the operators $I^u_k$ and evidently satisfy {\bf (H1)-(H3)}. In the next Lemma we state some measurability properties of operators \eqref{02-1},\eqref{02-2},\eqref{02-3},\eqref{02-5}. Additionally, properties {\bf (H4)-(H5)} for operators $\mathcal{A}^u_k$ are proved. 

\begin{lem}
\label{brownian_case} Let $k,k_1,\ldots,k_d\in\{1,\ldots,n\}.$

1) Given a measurable function $a:\mathcal{S}^{d}_+\times \mathbb{R}^n\to \mathbb{R},$ such that
$$
\forall u\in \mathbb{R}^n \ \ \ a(\cdot,u)\in L^2(\mathcal{S}^{d}_+),
$$
a family $\{I^u_{k_1,\ldots,k_d}(a(\cdot,u))\}_{u\in \mathbb{R}^n}$ can be realized as a measurable function on $\mathcal{C}^n\times \mathbb{R}^n$ w.r.t. the family $\{\mu^u\}_{u\in \mathbb{R}^n}.$

2) Given a measurable function $g:\mathcal{C}^n\times\mathbb{R}^n\to \mathbb{R},$ such that
$$
\forall u\in \mathbb{R}^n \ \ \ g(\cdot,u)\in L^2(\mathcal{C}^n,\mathcal{B}^n,\mu^u),
$$
a family $\{Q^u_{k_1,\ldots,k_d}(g(\cdot,u))\}_{u\in \mathbb{R}^n}$  can be realized as a measurable function on $\mathcal{S}^{d}_+\times \mathbb{R}^n$ w.r.t. the Lebesgue measure on $\mathcal{S}^d_+.$

3) Given a $\mathcal{P}\times \mathcal{B}(\mathbb{R}^n)-$measurable $a:\mathbb{R}_+\times \mathcal{C}^n \times \mathbb{R}^n\to \mathbb{R}$ such that
$$
\forall u\in \mathbb{R}^n \ \ \ a(\cdot,u)\in L^2(\mathbb{R}_+\times \mathcal{C}^n,\mathcal{P},dt\times \mu^u(d\omega^n)),
$$
a family $\{\mathcal{I}^u_k(a(\cdot,u))\}_{u\in \mathbb{R}^n}$ can be realized as a measurable function on the space $\mathcal{C}^n\times \mathbb{R}^n$ w.r.t. the family $\{\mu^u\}_{u\in \mathbb{R}^n}.$

4) Given a measurable $g:\mathcal{C}^n\times \mathbb{R}^n\to \mathbb{R}$ such that
$$
\forall u\in \mathbb{R}^n \ \ \ g(\cdot,u)\in L^2(\mathcal{C}^n,\mathcal{B}^n,\mu^u),
$$
a family $\{\mathcal{Q}^u_k(g(\cdot,u))\}_{u\in \mathbb{R}^n}$ can be realized as a
$\mathcal{P}\times \mathcal{B}(\mathbb{R}^n)-$measurable function on
$\mathbb{R}_+\times \mathcal{C}^n \times \mathbb{R}^n$ w.r.t. the family $\{dt\times \mu^u \}_{u\in \mathbb{R}^n}$ of measures on $\mathbb{R}_+\times \mathcal{C}^n.$
\end{lem}

\begin{proof}
Consider translations $\theta_u(\omega)=u+\omega,$ so that $\mu^u=\mu^0\circ \theta^{-1}_u.$ If $a:\mathcal{S}^{d}_+\to \mathbb{R}$ is a bounded Borel function  with compact support, then $I^u_{k_1,\ldots,k_d}a=(I^0_{k_1,\ldots,k_d}a) \circ \theta_{-u},$ $\mu^u-$a.s. Hence,  $(I^0_{k_1,\ldots,k_d}a)(\omega^n-u)$ is the needed measurable realization of the family $\{I^u_{k_1,\ldots,k_d}a\}_{u\in \mathbb{R}^n}.$
With the help of the Lemma \ref{measurability} the obtaines result is immediately generalized to any function $a:\mathcal{S}^{d}_+\times \mathbb{R}^n\to \mathbb{R},$ satisfying conditions of 1). Due to the Lemma \ref{measurability}, it is enough to prove 2) for a bounded Borel function $g:\mathcal{C}^n\to \mathbb{R}.$ Consider correspondence $u\to g(\cdot+u)$ as a measurable mapping of $\mathbb{R}^n$ into $L^2(\mathcal{C}^n,\mathcal{B}^n,\mu^0).$ Respectively, $u\to Q^0_{k_1,\ldots,k_d} (g(\cdot+u))$ is a measurable mapping of $\mathbb{R}^n$ into $L^2(\mathcal{S}^{d}_+).$ Using the Lemma \ref{measurability1}, it can be realized as a measurable function $h_{k_1,\ldots,k_d}:\mathcal{S}^d_+\times \mathbb{R}^n\to \mathbb{R}.$ Then
$$
g(\cdot+u)=\sum_{(k_1,\ldots,k_d)\in\mathcal{K}_n} I^0_{k_1,\ldots,k_d}  (h_{k_1,\ldots,k_d}(\cdot,u)).
$$
It follows that in $L^2(\mathcal{C}^n,\mathcal{B}^n,\mu^u)$
$$
g=\sum_{(k_1,\ldots,k_d)\in\mathcal{K}_n} (I^0_{k_1,\ldots,k_d} (h_{k_1,\ldots,k_d}(\cdot,u)))\circ \theta_{-u}=$$
$$
=\sum_{(k_1,\ldots,k_d)\in\mathcal{K}_n} I^u_{k_1,\ldots,k_d} (h_{k_1,\ldots,k_d}(\cdot,u)),
$$
i.e. $h_{k_1,\ldots,k_d}$ is a measurable realization of the family $\{Q^u_{k_1,\ldots,k_d} g\}_{u\in \mathbb{R}^n}.$

Proofs of properties 3) and 4) follow the same scheme.
\end{proof}

In the next Lemma measurability properties of operators \eqref{31_03_11}, \eqref{3112} are stated. Its proof reduces to the multiple applications of Lemmata \ref{brownian_case} and \ref{measurability}.

\begin{lem}
\label{measurability_stopped}
1) Given a measurable function $a:\mathcal{S}^{|k|}_+ \times \mathbb{R}^{n-2}_+\times\mathcal{S}^n\to \mathbb{R},$ such that
$$
\forall t_2,\ldots,t_{n-1},u \ \ a(\cdot,t_2,\ldots,t_{n-1},u)\in L^2(\mathcal{S}^{|k|}_+,\alpha_{\aleph_{t_2,\ldots,t_{n-1}},\mathcal{S}^n}(t_{|k|},u)dt),
$$
a family $\{(\mathcal{J}^{u,\aleph_{t_2,\ldots,t_{n-1}},\mathcal{S}^n}_k(a(\cdot,t_2,\ldots,t_{n-1},u)))\circ \mathcal{G}^{t_2,\ldots,t_{n-1}}\}_{t_2,\ldots,t_{n-1},u}$ can be realized as a measurable function on $\mathcal{C}^n \times \mathbb{R}^{n-2}_+\times\mathcal{S}^n$ w.r.t. the family $\{\varkappa^{t_2,\ldots,t_{n-1};u}\}_{t_2,\ldots,t_{n-1},u}.$

2) Given a measurable function $g:\mathcal{C}^n \times \mathbb{R}^{n-2}_+\times\mathcal{S}^n\to \mathbb{R},$ such that
$$
\forall t_2,\ldots,t_{n-1},u \ \ g(\cdot,t_2,\ldots,t_{n-1},u)\in L^2(\mathcal{C}^n,\mathcal{B}^n_{\tau_{\mathcal{S}^n}},\varkappa^{t_2,\ldots,t_{n-1};u}),
$$
a family $\{\mathcal{R}^{ t_2,\ldots,t_{n-1};u}_k(g(\cdot,t_2,\ldots,t_{n-1},u))\}_{t_2,\ldots,t_{n-1},u}$ can be realized as a measurable function on $\mathcal{S}^{|k|}_+ \times \mathbb{R}^{n-2}_+\times\mathcal{S}^n$ with respect to the family of measures $\{\alpha_{\aleph_{t_2,\ldots,t_{n-1}},\mathcal{S}^n}(t_{|k|},u)dt\}_{t_2,\ldots,t_{n-1},u}.$
\end{lem}

\begin{rem} Together with the Lemma \ref{measurability}, this result imply that $\mathcal{T}a$ is a measurable function. From \eqref{rho_eq} it follows that $\mathcal{T}$ is the isometry of the space $L^2(\mathcal{S}^{|k^1|,\ldots,|k^{n}|}_+,
\rho_{t^2_{|k^2|},\ldots,t^{n}_{|k^{n}|}}(u)dt^1\ldots dt^{n})$ into the space $L^2(\mathcal{S}^{|k^1|,\ldots,|k^{n_1}|}_+\times \mathcal{C}^{n},\rho_{t^2_{|k^2|},\ldots,t^{n-1}_{|k^{n-1}|}}(p^{n-1}(\omega^{n}))  dt^1\ldots dt^{n-1}\mu^u(d\omega^{n})).$
\end{rem}

\begin{thm}
\label{finite-dimensional}
$$
\mathcal{A}^u_{k^1,\ldots,k^{n}}: L^2(\mathcal{S}^{|k^1|,\ldots,|k^{n}|}_+,\rho_{t^2_{|k^2|},\ldots,t^{n}_{|k^{n}|}}(u)dt^1\ldots dt^{n}) \to L^2(\mathcal{C}^{n},\mathcal{B}^{n},\nu^u)
$$
are well-defined operators and all the hypotheses {\bf (H1)-(H5)} hold for a family $(\mathcal{A}^u_{k^1,\ldots,k^{n}}).$
\end{thm}

\begin{proof}
The property {\bf (H4)} of $\{\mathcal{A}^v_{k^1,\ldots,k^{n-1}}\}_{v\in\mathcal{S}^{n-1}}$ and Lemma \ref{measurability} imply that $\mathcal{A}^u_{k^1,\ldots,k^{n}}a$ is a measurable function of $(\omega^{n-1},\omega^n).$ Also, when $a$ is a bounded function of compact support, the property {\bf (H4)} for $\{\mathcal{A}^u_{k^1,\ldots,k^{n}}a\}_{u\in\mathcal{S}^{n}}$ immediately follows.

Properties {\bf (H1), (H2)} follow from the next calculation.
$$
\int_{\mathcal{C}^{n}}\int_{\mathcal{C}^{n-1}} \mathcal{A}^u_{k^1,\ldots,k^{n}}a(\omega^{n-1},\omega^{n})^2 \nu^{p^{n-1}(\omega^{n})}(d\omega^{n-1})\mu^u(d\omega^{n})=
$$
$$
\int_{\mathcal{C}^{n}}\bigg(\int_{\mathcal{C}^{n-1}} \mathcal{A}^{p^{n-1}(\omega^{n})}_{k^1,\ldots,k^{n-1}}(\mathcal{T}a(\cdot,\omega^{n}))(\omega^{n-1})
^2 \nu^{p^{n-1}(\omega^{n})}(d\omega^{n-1})\bigg)\mu^u(d\omega^{n})=
$$
$$
\int_{\mathcal{S}^{|k^1|,\ldots,|k^{n-1}|}_+\times\mathcal{C}^n} \mathcal{T}a(t^1,\ldots,t^{n-1},\omega^{n}))^2 \rho_{t^2_{|k^2|},\ldots,t^{n-1}_{|k^{n-1}|}}(p^{n-1})dt^1\ldots dt^{n-1}d\mu^u=
$$
$$
\int_{\mathcal{S}^{|k^1|,\ldots,|k^{n}|}_+} a(t^1,\ldots,t^{n})^2 \rho_{t^2_{|k^2|},\ldots,t^{n}_{|k^{n}|}}(u)dt^1\ldots dt^{n}.
$$
Next we prove that the Hilbert sum of all the spaces
$$
 \mathcal{A}^u_{k^1,\ldots,k^{n}}(L^2(\mathcal{S}^{|k^1|,\ldots,|k^{n}|}_+,\rho_{t^2_{|k^2|},\ldots,t^{n}_{|k^{n}|}}(u)dt^1\ldots dt^{n}))
$$
coincides with $ L^2(\mathcal{C}^{n},\mathcal{B}^{n},\nu^u).$
Consider $f\in L^2(\mathcal{C}^{n},\mathcal{B}^{n},\nu^u).$  As
$$
\mathbb{E}^{\nu^u}f^2=\int_{\mathcal{C}^{n}}\int_{\mathcal{C}^{n-1}} f(\omega^{n-1},\omega^{n})^2 \nu^{p^{n-1}(\omega^{n})}(d\omega^{n-1})\mu^u(d\omega^{n})<\infty,
$$
it follows that $f$ has a version such that $f(\cdot,\omega^{n})\in L^2(\mathcal{C}^{n-1},\mathcal{B}^{n-1},\nu^{p^{n-1}(\omega^{n})})$ for all $\omega^{n}.$
The inductive assumption imply that $f(\cdot,\omega^{n})$ has a series representation
\begin{equation}
\label{last2}
f(\cdot,\omega^{n})=\sum_{(k^1,\ldots,k^{n-1})\in\mathcal{K}_{1,\ldots ,n-1}} \mathcal{A}^{p^{n-1}(\omega^{n})}_{k^1,\ldots,k^{n-1}}(a_{k_1,\ldots k_{n-1}}(\cdot, \omega^{n})),
\end{equation}
where for each $\omega^n,$
$$
a_{k_1,\ldots, k_{n-1}}(\cdot,\omega^n)\in L^2(\mathcal{S}^{|k^1|,\ldots,|k^{n-1}|}_+,\rho_{t^2_{|k^2|},\ldots,t^{n-1}_{|k^{n-1}|}}(p^{n-1}(\omega^{n}))dt^1\ldots dt^{n-1}).
$$
Note that $a_{k_1,\ldots k_{n-1}}(\cdot,\omega^n)=P^{p^{n-1}(\omega^{n})}_{k_1,\ldots k_{n-1}}(f(\cdot,\omega^{n})).$ Property {\bf (H5)} of the induction hypothesis for  $\{\mathcal{A}^v_{k^1,\ldots,k^{n-1}}\}$ and Lemma \ref{measurability} imply that it is possible to choose functions $a_{k_1,\ldots k_{n-1}}$ measurable in all arguments. For fixed $(t^1,\ldots,t^{n-1}),$
$$
a_{k_1,\ldots, k_{n-1}}(t^1,\ldots,t^{n-1},\cdot)\in L^2(\mathcal{C}^n,\mathcal{B}^n_{\tau_{\mathcal{S}^n}},\varkappa^{u;t^2_{|k^2|},\ldots,t^{n-1}_{|k^{n-1}|}}),
$$
as the calculation below shows.
$$
\mathbb{E}^{\nu^u}f^2=\sum_{(k^1,\ldots,k^{n-1})\in\mathcal{K}_{1,\ldots ,n-1}}
\int_{\mathcal{C}^n}\int_{\mathcal{S}^{|k^1|,\ldots,|k^{n-1}|}_+}a_{k_1,\ldots, k_{n-1}}(t^1,\ldots,t^{n-1},\omega^{n})^2
$$
$$
\rho_{t^2_{|k^2|},\ldots,t^{n-1}_{|k^{n-1}|}}(p^{n-1}(\omega^{n}))
dt^1\ldots dt^{n-1}\mu^u(d\omega^{n})=
$$
$$
=\sum_{(k^1,\ldots,k^{n-1})\in\mathcal{K}_{1,\ldots ,n-1}}
\int_{\mathcal{S}^{|k^1|,\ldots,|k^{n-1}|}_+} \beta_{t^2_{|k^2|},\ldots,t^{n-1}_{|k^{n-1}|}}(u)
$$
$$
\int_{\mathcal{C}^n} a_{k_1,\ldots, k_{n-1}}(t^1,\ldots,t^{n-1},\omega^{n})^2 \varkappa^{t^2_{|k^2|},\ldots,t^{n-1}_{|k^{n-1}|};u}(d\omega^n)dt^1\ldots dt^{n-1}.
$$
From the Lemma \ref{add1} it follows that each  $a_{k^1,\ldots,k^{n-1}}$ can be represented as a sum 
\begin{equation}
\label{last1}
\begin{gathered}
a_{k^1,\ldots,k^{n-1}}(t^1,\ldots,t^{n-1},\cdot)= \\
=\sum_{k^{n}\in\mathcal{K}_{n}} (\mathcal{J}^{u,\aleph_{t^2_{|k^2|},\ldots,t^{n-1}_{|k^{n-1}|}},\mathcal{S}^n}_{k^{n}}(a_{k^1,\ldots,k^{n}}(t^1,\ldots,t^{n-1},\cdot)))\circ\mathcal{G}^{t^2_{|k^2|},\ldots,t^{n-1}_{|k^{n-1}|}}.
\end{gathered}
\end{equation}
Here, 
$$
a_{k^1,\ldots,k^{n}}(t^1,\ldots,t^{n-1},\cdot)=\mathcal{R}^{t^2_{|k^2|},\ldots,t^{n-1}_{|k^{n-1}|};u}_{k^n}(a_{k^1,\ldots,k^{n-1}}(t^1,\ldots,t^{n-1},\cdot))
$$ 
is mea\-su\-rable in all arguments by the Lemma \ref{measurability_stopped}. In fact,
$$
a_{k^1,\ldots,k^{n}} \in L^2(\mathcal{S}^{|k^1|,\ldots,|k^{n}|}_+,\rho_{t^2_{|k^2|},\ldots,t^{n}_{|k^{n}|}}(u)dt^1\ldots dt^{n}),
$$
as follows from the identity
$$
\mathbb{E}^{\nu^u}f^2=\sum_{(k^1,\ldots,k^{n})\in\mathcal{K}_{1,\ldots ,n}}
\int_{\mathcal{S}^{|k^1|,\ldots,|k^{n-1}|}_+} \int_{\mathcal{S}^{|k^n|}_+}
$$
$$
\beta_{t^2_{|k^2|},\ldots,t^{n-1}_{|k^{n-1}|}}(u) \alpha_{\aleph_{t^2_{|k^2|},\ldots,t^{n-1}_{|k^{n-1}|}},\mathcal{S}^n}(t^n_{|k^n|},u) a_{k^1,\ldots,k^{n}}(t^1,\ldots,t^{n})^2 dt^1\ldots dt^n
$$
and the definition of functions $\rho$ \eqref{rho_eq}. Hence, $\mathcal{A}^u_{k^1,\ldots,k^n} a_{k^1,\ldots,k^n}$ are well-defined and the series
$$
\sum_{(k^1,\ldots,k^n)\in \mathcal{K}_{1,\ldots ,n}}\mathcal{A}^u_{k^1,\ldots,k^n} a_{k^1,\ldots,k^n}
$$
converges. It remains to check that its sum equals $f.$

Denote $f_{k^1,\ldots,k^{n-1}}(\omega^{n-1},\omega^n)=\mathcal{A}^{p^{n-1}(\omega^n)}_{k^1,\ldots,k^{n-1}} (a_{k^1,\ldots,k^{n-1}}(\cdot,\omega^n))(\omega^{n-1}).$ It follows from \eqref{last1} and the definition of $\mathcal{I}$ \eqref{I_def}, that
$$
a_{k^1,\ldots,k^{n-1}}(t^1,\ldots,t^{n-1},\cdot)=\sum_{k^n\in \mathcal{K}_n} \mathcal{T}a_{k^1,\ldots,k^{n}}(t^1,\ldots,t^{n-1},\cdot)
$$
in $L^2(\mathcal{C}^n,\mathcal{B}^n_{\tau_{\mathcal{S}^n}},\varkappa^{t^2_{|k^2|},\ldots,t^{n-1}_{|k^{n-1}|};u}).$ A straightforward calculation implies that
$$
f_{k^1,\ldots,k^{n-1}}=\sum_{k^n\in \mathcal{K}_n} \mathcal{A}^u_{k^1,\ldots,k^n} a_{k^1,\ldots,k^n}.
$$
Hence, the needed conclusion will follow from
$$
f=\sum_{(k^1,\ldots,k^{n-1})\in \mathcal{K}_{1,\ldots ,n-1}}f_{k^1,\ldots,k^{n-1}},
$$
which in turn is a consequence of \eqref{last2}.

A property {\bf (H5)} for $\{\mathcal{A}^u_{k^1,\ldots,k^n}\}_{u\in \mathcal{S}^n}$ follows from the identity
\begin{equation}
\label{04-1}
P^u_{k^1,\ldots,k^n}f (t^1,\ldots,t^n)=\mathcal{R}^{t^2_{|k^2|},\ldots,t^{n-1}_{|k^{n-1}|};u}_{k^n} (P^{p^{n-1}(\omega^n)}_{k^1,\ldots,k^{n-1}}(f(\cdot,\omega^n))),
\end{equation}
obtained during the proof.

\end{proof}

\begin{rem}
Theorems \ref{stopped} and \ref{finite-dimensional} not only state the existence of analogues of the It\^o-Wiener expansion for stopped Brownian motion and for $n-$point motions of the Arratia flow, but also reduce the calculation of these analogues to the calculation of the It\^o-Wiener expansion in the Gaussian case, as follows from \eqref{04-1}, \eqref{04-2}.
\end{rem}

\section{Two results on measurable realizations}

The first result is a variant of a measurable selection theorem and seems to be known. Still, we were not able to find a correct reference, so we provide a proof here. In the first Lemma we consider the space $L^0$ of all measurable functions on the measure space $(\mathcal{X},\mathcal{B},\mu)$ equipped with the distance  $d_0(\xi_1,\xi_2)=\mathbb{E}^{\mu}\min(|\xi_1-\xi_2|,1).$

\begin{lem}
\label{measurability1}
Let $(\Omega,\mathcal{F})$ be a measurable space. For any measurable mapping 
$\eta:\Omega \to L^0$ with separable range $\eta(\Omega)\subset L^0,$
a family $\{\eta(\omega)\}_{\omega\in \Omega}$ can be realized as a measurable function on $\mathcal{X}\times\Omega$ w.r.t. the measure $\mu.$
\end{lem}

\begin{proof}
Let $(\xi_n)_{n\geq 1}$  be a dense sequence in $\eta(\Omega).$ Then for each $k\geq 0$
$$
\eta(\Omega)\subset \bigcup_{n\geq 1} \Delta^{(k)}_n,
$$
where $\Delta^{(k)}_n=B(\xi_n,2^{-k})\setminus \bigcup_{1\leq m<n}B(\xi_m,2^{-k})$ and $B(\xi,r)$ is an open ball in $(L^0,d_0).$

Define $\eta_k(x,\omega)=\sum_{n\geq 1} \xi_n(x) 1_{\eta(\omega)\in \Delta^{(k)}_n}.$ Each $\eta_k$ is a measurable function on $\mathcal{X}\times \Omega.$ From inequality
$$
\sup_{\omega\in\Omega} d_0(\eta(\omega),\eta_k(\cdot,\omega))\leq 2^{-k}.
$$
it follows that for each $\omega\in\Omega$ $\eta_k(\cdot,\omega)\to \eta(\omega)$ a.s., $k\to \infty.$ Hence,
$$
\tilde{\eta}(x,\omega)=
\begin{cases}
\lim_{k\to\infty} \eta_k(x,\omega), \ \mbox{the limit exists} \\
0, \ \mbox{otherwise}
\end{cases}
$$
is a measurable realization of $\{\eta(\omega)\}_{\omega\in \Omega}.$
\end{proof}

The main distinction of the next Lemma from usual measurable selection theorems is that it describes measurable realizations for families of measurable functions with values in different measure spaces.

\begin{lem}
\label{measurability}
Let $(\mathcal{X},\mathcal{B}_{\mathcal{X}}),$ $(\mathcal{Y},\mathcal{B}_{\mathcal{Y}}),$  $(\Omega,\mathcal{F})$ be measurable spaces,
$(\mu^\omega)_{\omega\in\Omega},$ $(\nu^\omega)_{\omega\in\Omega}$ be regular families of measures on $(\mathcal{X},\mathcal{B}_{\mathcal{X}}),$ $(\mathcal{Y},\mathcal{B}_{\mathcal{Y}}),$ respectively. For each $\omega\in \Omega$ let
$$
\mathcal{A}^\omega:L^2(\mathcal{X},\mathcal{B}_{\mathcal{X}},\mu^\omega)\to L^2(\mathcal{Y},\mathcal{B}_{\mathcal{Y}},\nu^\omega)
$$
be a bounded linear operator. Assume that

1) there exist sequences of measurable functions $f_n:\mathcal{X}\to \mathbb{R},$ $k_n:\mathcal{Y}\to \mathbb{R}$ such that for every $\omega\in\Omega$ the sequence $(f_n)_{n\geq 1}$ is total in $L^2(\mathcal{X},\mathcal{B}_{\mathcal{X}},\mu^\omega)$ and the sequence $(k_n)_{n\geq 1}$ is total in $L^2(\mathcal{Y},\mathcal{B}_{\mathcal{Y}},\nu^\omega);$

2) for each $n\geq 1$ a family $\{\mathcal{A}^\omega f_n\}_{\omega\in\Omega}$ can be realized as a measurable function on $\mathcal{Y}\times\Omega$ w.r.t. the family $\{\nu^\omega\}_{\omega\in \Omega}.$

Then given a measurable $g:\mathcal{X}\times \Omega\to \mathbb{R}$ such that $\forall\omega\in\Omega$ $g(\cdot,\omega)\in L^2(\mathcal{X},\mathcal{B}_{\mathcal{X}},\mu^\omega),$ a family $\{\mathcal{A}^\omega (g(\cdot,\omega))\}_{\omega\in\Omega}$ can be realized as a measurable function on $\mathcal{Y}\times\Omega$ w.r.t. the family $\{\nu^\omega\}_{\omega\in \Omega}.$
\end{lem}
\begin{proof}

At first we construct a suitable family of orthonormal bases in $L^2(\mathcal{X},\mathcal{B}_\mathcal{X},\mu^\omega).$ We apply the Gram--Schmidt orthonormalization procedure to a sequence $(f_n)_{n\geq 1},$ whose existence is stipulated in the conditions of the Lemma, i.e. put
$$
e_1(x,\omega)=\bigg(\int_\mathcal{X} f^2_1 d\mu^\omega\bigg)^{-1/2} f_1(x)1_{\{\int_\mathcal{X} f^2_1 d\mu^\omega>0\}};
$$
$$
e_n(x,\omega)=\bigg(\int_\mathcal{X} f^2_n d\mu^\omega-\sum^{n-1}_{k=1}(\int_\mathcal{X} f_n e_k(\cdot,\omega) d\mu^\omega)^2\bigg)^{-1/2} (f_n(x)-
$$
$$
-\sum^{n-1}_{k=1}(\int_\mathcal{X} f_n e_k(\cdot,\omega) d\mu^\omega)e_k(\cdot,\omega)
)
1_{\{\int_\mathcal{X} f^2_n d\mu^\omega>\sum^{n-1}_{k=1}(\int_\mathcal{X} f_n e_k(\cdot,\omega) d\mu^\omega)^2\}}.
$$
As a result we obtain functions $e_n$ of the form
\begin{equation}
\label{aux11}
e_n(x,\omega)=\sum^n_{k=1}c_{n,k}(\omega)f_k(x),
\end{equation}
such that for each $\omega$  a set $\{e_n(\cdot,\omega), \int_{\mathcal{X}}e_n(x,\omega)^2\mu^\omega(dx)>0\}$ is an orthonormal basis in $L^2(\mathcal{X},\mathcal{B}_\mathcal{X},\mu^\omega).$

The same considerations with $(k_n)_{n\geq 1}$ imply that there exists a sequence of measurable functions $j_n:\mathcal{Y}\times\Omega\to \mathbb{R},$ such that for each $\omega$  a set $\{j_n(\cdot,\omega), \int_{\mathcal{Y}}j_n(y,\omega)^2\nu^\omega(dy)>0\}$ is an orthonormal basis in $L^2(\mathcal{Y},\mathcal{B}_\mathcal{Y},\nu^\omega).$

Consider a measurable $g:\mathcal{X}\times \Omega\to \mathbb{R}$ such that $\forall\omega\in\Omega$ $g(\cdot,\omega)\in L^2(\mathcal{X},\mathcal{B}_{\mathcal{X}},\mu^\omega).$ Then  in $L^2(\mathcal{Y},\mathcal{B}_{\mathcal{Y}},\nu^\omega),$
$$
\mathcal{A}^\omega (g(\cdot,\omega)) =\sum^{\infty}_{n=1}
\bigg(\int_{\mathcal{X}} g(\cdot,\omega)e_n(\cdot,\omega) d\mu^\omega\bigg)\mathcal{A}^\omega (e_n(\cdot,\omega))
$$
From \eqref{aux11} and the assumption of the Lemma,  each family of functions $\{\mathcal{A}^\omega (e_n(\cdot,\omega))\}_{\omega\in\Omega}$ can be realized as a measurable function on $\mathcal{Y}\times\Omega.$ Integrating $\mathcal{A}^\omega (g(\cdot,\omega))$ with $j_k(\cdot,\omega)$ gives
$$
\int_{\mathcal{Y}} \mathcal{A}^\omega (g(\cdot,\omega)) j_k(\cdot,\omega) d\nu^\omega=
\sum^{\infty}_{n=1}
\bigg(\int_{\mathcal{X}} g(\cdot,\omega)e_n(\cdot,\omega) d\mu^\omega\bigg) \times
$$
$$
\times
\bigg(\int_{\mathcal{Y}}\mathcal{A}^\omega (e_n(\cdot,\omega))j_k(\cdot,\omega) d\nu^\omega\bigg).
$$
Consequently, the mapping $\omega\to \int_{\mathcal{Y}} \mathcal{A}^\omega (g(\cdot,\omega)) j_k(\cdot,\omega) d\nu^\omega$ is measurable. Define
$$
h_n(y,\omega)=\sum^n_{k=1} \bigg(\int_{\mathcal{Y}} \mathcal{A}^\omega (g(\cdot,\omega)) j_k(\cdot,\omega) d\nu^\omega\bigg) j_k(y,\omega).
$$
For each $l\geq 1$ define
$$
n_l(\omega)=\min\{n\geq 1: \ \sum_{k>n}\bigg(\int_{\mathcal{Y}} \mathcal{A}^\omega (g(\cdot,\omega))j_k(\cdot,\omega) d\nu^\omega\bigg)^2<2^{-l}\}.
$$
Every function $h_{n_l(\omega)}(y,\omega)$ is measurable and
$$
\nu^\omega\{|\mathcal{A}^\omega (g(\cdot,\omega))-h_{n_l(\omega)}(\cdot,\omega)|>1/l\}\leq l^2 2^{-l}.
$$
Hence, for each $\omega,$
$$
h_{n_l(\omega)}(\cdot,\omega)\to \mathcal{A}^\omega (g(\cdot,\omega)), \ \nu^\omega-\mbox{a.s.}
$$
and the function 
$$
\tilde{h}(y,\omega)=
\begin{cases}
\lim_{l\to\infty} h_{n_l(\omega)}(y,\omega), \ \mbox{the limit exists} \\
0, \ \mbox{otherwise}
\end{cases}
$$
is a mesurable realization of  $\{\mathcal{A}^\omega (g(\cdot,\omega))\}_{\omega\in\Omega}.$
\end{proof}

\end{document}